\documentclass[11pt]{amsart}
\usepackage{amsmath, amsthm, amssymb}
\usepackage{graphicx}
\usepackage[usenames]{color}
\usepackage{srcltx} 
\usepackage{verbatim}
\usepackage{float}

\usepackage{tikz}

\newtheorem{thm}{Theorem}[section]
\newcommand{\bt}{\begin{thm}}
\newcommand{\et}{\end{thm}}

\newtheorem{ex}[thm]{Example}
\newtheorem{Cor}[thm]{Corollary}   
\newtheorem{lem}[thm]{Lemma}   
\newtheorem{prop}[thm]{Proposition}
\newtheorem{Def}[thm]{Definition}
\newtheorem{rmrk}[thm]{Remark}

\newcommand{\GHto}{\stackrel { \textrm{GH}}{\longrightarrow} }
\newcommand{\Fto}{\stackrel {\mathcal{F}}{\longrightarrow} }

\newcommand{\set}{\rm{set}}
\newcommand{\mass}{{\mathbf M}}

\newcommand{\N}{\mathbb{N}}

\newcommand{\R}{\mathbb{R}}

\def\Xint#1{\mathchoice
{\XXint\displaystyle\textstyle{#1}}%
{\XXint\textstyle\scriptstyle{#1}}%
{\XXint\scriptstyle\scriptscriptstyle{#1}}%
{\XXint\scriptscriptstyle\scriptscriptstyle{#1}}%
\!\int}
\def\XXint#1#2#3{{\setbox0=\hbox{$#1{#2#3}{\int}$ }
\vcenter{\hbox{$#2#3$ }}\kern-.6\wd0}}

\def\dashint{\Xint-}

\begin{document}

\title[IMCF and the Stability of the PMT]{Inverse Mean Curvature Flow and the Stability of the Positive Mass Theorem}

\author{Brian Allen}
\address{University of Hartford}
\email{brianallenmath@gmail.com}

\date{Fall 2021}


\begin{abstract}
We study the stability of the Positive Mass Theorem (PMT) in the case where a sequence of regions of manifolds with positive scalar curvature $U_T^i\subset M_i^3$ are foliated by a smooth solution to Inverse Mean Curvature Flow (IMCF) which may not be uniformly controlled near the boundary. Then if $\partial U_T^i = \Sigma_0^i \cup \Sigma_T^i$, $m_H(\Sigma_T^i) \rightarrow 0$ and extra technical conditions are satisfied we show that $U_T^i$ converges to a flat annulus with respect to Sormani-Wenger Intrinsic Flat (SWIF) convergence. 
\end{abstract}

\maketitle

\section{Introduction}\label{sect-Intro}
 If we consider a complete, asymptotically flat manifold with nonnegative scalar curvature $M^3$ then the Positive Mass Theorem (PMT) says that $M^3$ has positive ADM mass. The rigidity statement says that if $m_{ADM}(M) = 0$ then $M$ is isometric to Euclidean space. In this paper we are concerned with the stability of this rigidity statement in the case where we can foliate a region of $M$ by a smooth solution of Inverse Mean Curvature Flow (IMCF).

The stability problem for the PMT has been studied by many authors and one should see the author's previous work \cite{BA2} for a more complete discussion of the history of this problem. Here we particularly note the work of Lee and Sormani \cite{LS1} on stability of the PMT under the assumption of rotationally symmetry and the work of Huang, Lee and Sormani \cite{HLS} under the assumption that the asymptotically flat manifold can be represented as a graph over $\R^n$. We also note the recent work of Sormani and Stavrov \cite{SS} where the stability of the PMT is studied on geometricstatic manifolds and the work of Bryden \cite{Br} where stability of the PMT is studied on axisymmetric manifolds under $W^{1,p}$, $1 \le p < 2$, convergence.

In \cite{BA2} the author studied the stability of the PMT on manifolds which can be foliated by a smooth solution of Inverse Mean Curvature Flow (IMCF) which is uniformly controlled. Under these assumption the author was able to show that a sequence of regions of asymptotically flat manifolds whose Hawking mass goes to zero will converge to Euclidean space under $L^2$ convergence. The goal of the current paper is to extend these results in order to address the conjecture of Lee and Sormani \cite{LS1} on the stability of the PMT under Sormani-Wenger Intrinsic Flat(SWIF) convergence. 

In \cite{HI}, Huisken and Ilmanen show how to use weak solutions of IMCF in order to prove the Riemannian Penrose Inequality (RPI) in the case of a connected boundary and they note that their techniques give another proof of the  PMT for asymptotically flat Riemanian manifolds when $n=3$ (see Schoen and Yau \cite{SY}, and Witten \cite{W} for more general proofs of the PMT as well as Bray \cite{B} for a more general proof of the RPI). The rigidity of both the PMT and the RPI are also proved in \cite{HI} and the present work builds off of these arguments by using IMCF to provide a special coordinate system on each member of the sequence of manifolds $M^3_i$ to show stability of the PMT. 

If we have $\Sigma^2$ a surface in a Riemannian manifold, $M^3$, we will denote the induced metric, mean curvature, second fundamental form, principal curvatures, Gauss curvature, area, Hawking mass and Neumann isoperimetric constant as $g$, $H$, $A$, $\lambda_i$, $K$, $|\Sigma|$, $m_H(\Sigma)$, $IN_1(\Sigma)$, respectively. We will denote the Ricci curvature, scalar curvature, sectional curvature tangent to $\Sigma$, and ADM mass as $Rc$, $R$, $K_{12}$, $m_{ADM}(M)$, respectively.

Now the class of regions of manifolds to which we will by proving stability of the PMT is defined.

\begin{Def} \label{IMCFClass} Define the class of manifolds with boundary foliated by IMCF as follows
\begin{align*}
\mathcal{M}_{r_0,H_0,I_0}^{T,H_1,A_1}:=\{& U_T \subset M, R \ge 0|
\exists \Sigma \subset M \text{compact, connected surface such that } 
\\& IN_1(\Sigma) \ge I_0, m_H(\Sigma) \ge 0 \text{,and } |\Sigma|=4\pi r_0^2. 
\\ &\exists \Sigma_t \text{ smooth solution to IMCF, such that }\Sigma_0=\Sigma,
\\& H_0 \le H(x,t) \le H_1, |A|(x,t) \le A_1 \text{ for } t \in [0,T],
\\&\text{and } U_T = \{x\in\Sigma_t: t \in [0,T]\} \}
\end{align*}
where $0 < H_0 < H_1 < \infty$, $0 < I_0,A_1,r_0 < \infty$ and $0 < T < \infty$.
\end{Def}

Before we state the stability theorems we define some metrics on $\Sigma \times [0,T]$ for Riemmanian metrics $\delta, \hat{g}^i \in \mathcal{M}_{r_0,H_0,I_0}^{T,H_1,A_1}$ foliated by IMCF that will be used throughout this document:
\begin{align}
\delta &= \frac{r_0^2}{4}e^t dt^2 + r_0^2e^t \sigma
\\ \hat{g}^i&=\frac{1}{H_i(x,t)^2}dt^2 + g^i(x,t)
\end{align}
where $\sigma$ is the round metric on $\Sigma$ and $g^i(x,t)$ is the metric on $\Sigma_t^i$. The first metric is the flat Euclidean metric and the second is the metric on $U_T^i$ with respect to the IMCF foliation.

We now state our first result which assumes uniform control on various curvature quantities in order to obtain GH and SWIF convergence. In general we do not expect GH stability of the PMT (See Example 5.6 in Lee and Sormani \cite{LS1}) but the curvature bounds assumed make GH convergence reasonable in these theorems. It is important to note that we do not need bounds on the full Ricci tensor though and again the importance of this Theorem is when we allow these bounds to degenerate in order to prove Theorem \ref{SPMTJumpRegion}.

One should note that all of the norms in the theorems below are defined with respect to $\delta$ on $\Sigma\times[0,T]$. This requires a diffeomorphism onto the coordinate space which was defined in the author's previous work \cite{BA2} and is discussed in Theorem \ref{avgH}.

\begin{thm}\label{SPMT}
Let $U_{T}^i \subset M_i^3$ be a sequence s.t. $U_{T}^i\subset \mathcal{M}_{r_0,H_0,I_0}^{T,H_1,A_1}$ and $m_H(\Sigma_{T}^i) \rightarrow 0$ as $i \rightarrow \infty$.  If we assume that
\begin{align}
 |K^i|_{C^{0,\alpha}(\Sigma\times[0,T])} &\le C
\\|Rc^i(\nu,\nu)|_{C^{0,\alpha}(\Sigma\times[0,T])} &\le C
\\ |R^i|_{C^{0,\alpha}(\Sigma\times[0,T])} &\le C
 \\diam(\Sigma_t^i) &\le D \hspace{0.25cm} \forall i,\hspace{0.25cm} t\in [0,T]
 \\ H_i(x,0)^2 &\le \frac{4}{r_0^2} + \frac{C}{i}
 \end{align}
then $\hat{g}^i$ converges uniformly to $\delta$, as well as,
\begin{align}
(U_T,\hat{g}^i) &\GHto (S^2 \times [0,T]),\delta)
\\(U_T,\hat{g}^i) &\Fto (S^2 \times [0,T]),\delta).
\end{align}
\end{thm}
\begin{rmrk}\label{CompareToFinster}
One should compare this theorem with the work of Finster \cite{F}, Bray and Finster \cite{BF}, and Finster and Kath \cite{FK} in the case where $M^i$ is a spin manifold and a $C^0$ bound on $|Rm|$ and a $L^2$ bound on $|\nabla Rm|$ are assumed in order to show $C^0$ convergence to Euclidean space. In that case the level sets of the spinor field $\psi$ are analogous to the solution of IMCF at time $t$, $\Sigma_t$, and so it is interesting to note that we do not need a bound on the Riemann tensor in this work. Instead, if we settle for a weaker notion of convergence then we can get away with considerably weaker assumptions on the curvature of $M_i$.
\end{rmrk}
We now give one more version of the stability of the PMT where we require less curvature information in exchange for comparison inequality between the metric on $\Sigma_t^i$ induced from $U_T^i$ and the corresponding sphere $r_0^2\sigma e^t$.
\begin{thm}\label{SPMTLessCurv}
Let $U_{T}^i \subset M_i^3$ be a sequence s.t. $U_{T}^i\subset \mathcal{M}_{r_0,H_0,I_0}^{T,H_1,A_1}$ and $m_H(\Sigma_{T}^i) \rightarrow 0$ as $i \rightarrow \infty$.  If we assume that
\begin{align}
|Rc^i(\nu,\nu)|_{C^{0,\alpha}(\Sigma\times[0,T])} &\le C\label{RicciOnlyCurvBound}
 \\ \left(1 - \frac{C}{j} \right) r_0^2 e^t \sigma \le g^i(x,t) &\le C r_0^2 e^t \sigma \text{ } \forall (x,t) \in \Sigma\times [0,T]\label{MetricBounds}
 \\ H_i(x,0)^2 &\le \frac{4}{r_0^2} + \frac{C}{i}
 \end{align}
then $\hat{g}^i$ converges uniformly to $\delta$, as well as,
\begin{align}
(U_T,\hat{g}^i) &\GHto (S^2 \times [0,T]),\delta)
\\(U_T,\hat{g}^i) &\Fto (S^2 \times [0,T]),\delta).
\end{align}
\end{thm}
\begin{rmrk}
One can think of assumption \eqref{MetricBounds} as a relaxation of the assumption of rotational symmetry made by Lee and Sormani \cite{LS1}. The lower bound is stronger than the upper bound because of the observations of the author and Sormani \cite{BS} on comparing $L^2$ convergence to GH and SWIF convergence (See Theorem \ref{WarpConv} for a similar assumption). In particular, Example 3.4 of the author and Sormani's paper \cite{BS} illustrates what can happen if this lower bound is not assumed.

In this version of the stability of the PMT one could replace assumption \eqref{RicciOnlyCurvBound} with 
\begin{align}
Rc^i(\nu,\nu) \ge - \frac{C}{i}\label{RicciC0LowerBound}
\end{align}
and the result will also follow. This is becuase in the other versions of stability the assumption \eqref{RicciOnlyCurvBound} is used in a more essential way but here it is just used to achieve \eqref{RicciC0LowerBound} and hence it could be advantageous to allow the upper bound on $Rc^i(\nu,\nu)$ to degenerate.
\end{rmrk}
 
Now if we define 
\begin{align} 
 U_{t_1^k}^{i,t_2^k} = \{x \in \Sigma_t^i: t \in [t_1^k,t_2^k]\}
 \end{align}
then we can say that 
\begin{align}
U_{t_1^k}^{i,t_2^k}\in \mathcal{M}_{r_0,H_0,I_0}^{t_2^k-t_1^k,H_1,A_1}
\end{align} 
by noticing that the substitution $s = t-t_1^k$ where $s \in [0,t_2^k-t_1^k]$ implies 
\begin{align}
U_{t_2^k-t_1^k}^i\in \mathcal{M}_{r_0,H_0,I_0}^{t_2^k-t_1^k,H_1,A_1}.
\end{align}
Then we can obtain the new result which allows the regions foliated by IMCF to approach jump regions and obtains SWIF convergence (See figure \ref{DepictingSWIFStability}).

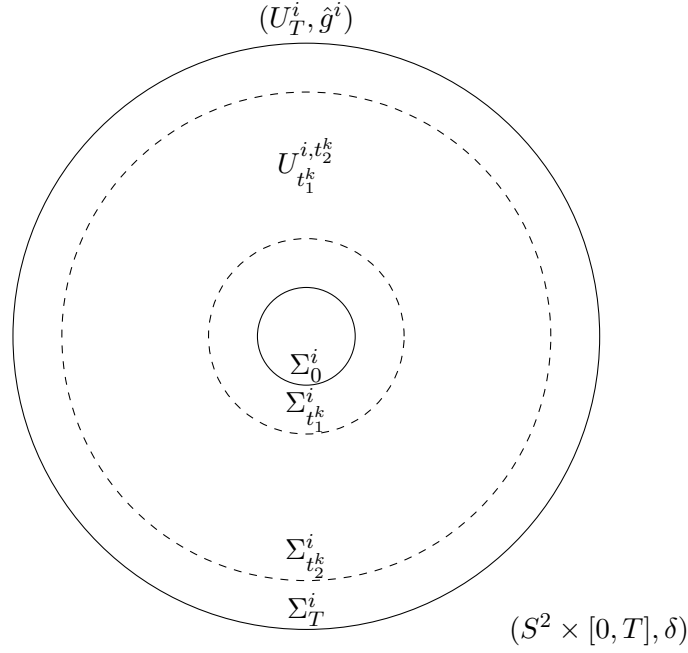
\begin{figure}\label{DepictingSWIFStability}
\begin{tikzpicture}[scale=.65]
\draw (0,0) circle (1cm);
\draw[dashed] (0,0) circle (2cm);
\draw[dashed] (0,0) circle (5cm);
\draw (0,0) circle (6cm);
\draw (0,6.5) node{$(U_T^i,\hat{g}^i)$};
\draw (0,3.5) node{$U_{t_1^k}^{i,t_2^k}$};
\draw (0,-.6) node{$\Sigma_0^i$};
\draw (0,-1.5) node{$\Sigma_{t_1^k}^i$};
\draw (0,-4.5) node{$\Sigma_{t_2^k}^i$};
\draw (0,-5.6) node{$\Sigma_T^i$};
\draw (6,-6) node{$(S^2\times [0,T],\delta) $};
\end{tikzpicture}
\caption{Visualizing the construction of Theorem \ref{SPMTJumpRegion} in the paramaterization space $S^2\times [0,T]$.}
\end{figure}

 \begin{thm}\label{SPMTJumpRegion}
Let $U_{T}^i \subset M_i^3$ be a sequence and choose $t_1^k,t_2^k \in [0,T]$, $k \in \N$ where  
\begin{align}
t_1^k < t_2^k, \hspace{1cm} \lim_{k \rightarrow \infty} t_1^k = 0, \hspace{1cm} \lim_{k \rightarrow \infty} t_2^k = T.
\end{align}
Assume $U_{t_1^k}^{i,t_2^k}\subset \mathcal{M}_{r_0^k,H_0^k,I_0^k}^{t_2^k-t_1^k,H_1^k,A_1^k}$ $ \forall k >0$ where
\begin{align}
\lim_{k \rightarrow \infty} r_0^k = r_0, \hspace{.5cm} \lim_{k \rightarrow \infty}H_0^k =\lim_{k \rightarrow \infty}I_0^k = 0 , \hspace{.5cm} \lim_{k \rightarrow \infty}H_1^k = \lim_{k \rightarrow \infty}A_1^k = \infty \label{HDegenerates}
\end{align}
and $m_H(\Sigma_{T}^i) \rightarrow 0$ as $i \rightarrow \infty$.  If we further assume that
\begin{align}
 |K^i|_{C^{0,\alpha}(\Sigma\times[t_1^k,t_2^k])} &\le C_k\label{GaussBound}
\\|Rc^i(\nu,\nu)|_{C^{0,\alpha}(\Sigma\times[t_1^k,t_2^k])} &\le C_k \label{RicciBound}
\\ |R^i|_{C^{0,\alpha}(\Sigma\times[t_1^k,t_2^k])} &\le C_k \label{ScalarBound}
\\\liminf_{i \rightarrow \infty} d_{\Sigma_0^i}(\theta_1,\theta_2) &\ge d_{S^2(r_0)}(\theta_1,\theta_2) \text{ } \forall  \theta_1,\theta_2 \in S^2 \label{distanceAssumption1}
 \\diam(\Sigma_t^i) &\le C \hspace{0.25cm} \forall i,\hspace{0.25cm} t\in [0,T] \label{diamBound}
 \\ H_i(x,t_1^k)^2 &\le \frac{4}{r_0^2} + \frac{C_k}{i} \label{HUpperBound}
 \\ \exists K \in \N, \forall k \ge K \text{ } A^i(\cdot,\cdot) &> 0 \text{ on } U_T^i \setminus (U_T^{i,k} \cup \partial U_T^i)\label{Convexity}
 \\ \exists f(t) \in L^1([0,T]) \text{ } \frac{1}{H_i(x,t)} &\le f(t) \text{ for } t \in [0,T] \label{HBlowUpRate}
 \end{align}
 where $\displaystyle \lim_{k \rightarrow \infty} C_k = \infty$ then
\begin{align}
(U_T,\hat{g}^i) &\Fto (S^2 \times [0,T],\delta).
\end{align}
\end{thm}

This result is in line with the conjecture of Lee and Sormani \cite{LS1} since it gives SWIF convergence only but now we would like to remove some of the curvature conditions of the last theorem in exchange for metric bounds on $g_i$ and so we replace \eqref{GaussBound} and \eqref{ScalarBound} with \eqref{sigmaMetricBound1}.

 \begin{thm}\label{SPMTJumpRegionWeakCurv}
Let $U_{T}^i \subset M_i^3$ be a sequence and choose $t_1^k,t_2^k \in [0,T]$, $k \in \N$ where  
\begin{align}
t_1^k < t_2^k, \hspace{1cm} \lim_{k \rightarrow \infty} t_1^k = 0, \hspace{1cm} \lim_{k \rightarrow \infty} t_2^k = T.
\end{align}
Assume $U_{t_1^k}^{i,t_2^k}\subset \mathcal{M}_{r_0^k,H_0^k,I_0^k}^{t_2^k-t_1^k,H_1^k,A_1^k}$ $ \forall k >0$ where
\begin{align}
\lim_{k \rightarrow \infty} r_0^k = r_0, \hspace{.5cm} \lim_{k \rightarrow \infty}H_0^k =\lim_{k \rightarrow \infty}I_0^k = 0 , \hspace{.5cm} \lim_{k \rightarrow \infty}H_1^k = \lim_{k \rightarrow \infty}A_1^k = \infty \label{HDegenerates}
\end{align}
and $m_H(\Sigma_{T}^i) \rightarrow 0$ as $i \rightarrow \infty$.  If we further assume that
\begin{align}
|Rc^i(\nu,\nu)|_{C^{0,\alpha}(\Sigma\times[t_1^k,t_2^k])} &\le C_k \label{RicciBound}
 \\\left(1 - \frac{C_k}{j} \right) r_0^2 e^t \sigma \le g^i(x,t) &\le C_k r_0^2 e^t \sigma \text{ } \forall (x,t) \in \Sigma\times [0,T] \label{sigmaMetricBound1}
 \\\liminf_{i \rightarrow \infty} d_{\Sigma_0^i}(\theta_1,\theta_2) &\ge d_{S^2(r_0)}(\theta_1,\theta_2) \text{ } \forall  \theta_1,\theta_2 \in S^2 \label{distanceAssumption1weakCurv}
 \\ H_i(x,t_1^k)^2 &\le \frac{4}{r_0^2} + \frac{C_k}{i} \label{HUpperBound}
 \\ \exists K \in \N, \forall k \ge K \text{ } A^i(\cdot,\cdot) &> 0 \text{ on } U_T^i \setminus (U_T^{i,k} \cup \partial U_T^i)\label{Convexity}
 \\ \exists f(t) \in L^1([0,T]) \text{ } \frac{1}{H_i(x,t)} &\le f(t) \text{ for } t \in [0,T] \label{HBlowUpRate}
 \end{align}
 where $\displaystyle \lim_{k \rightarrow \infty} C_k = \infty$ then
\begin{align}
(U_T,\hat{g}^i) &\Fto (S^2 \times [0,T],\delta).
\end{align}
\end{thm}

\begin{rmrk}\label{CompareToLeeSormani}
 One should compare this result with the result of Lee and Sormani \cite{LS1} where SWIF convergence is obtained under the assumption of rotational symmetry. In particular, notice that we allow the lower bound on mean curvature to degenerate \eqref{HDegenerates} which is what we expect when a thin gravity well develops at the center of the parameterization space. Also, notice that we allow the curvature bounds \eqref{RicciBound}, \eqref{ScalarBound} to degenerate which is why we do not expect GH convergence and which is also expected to happen when thin gravity wells develops at the boundary of the parameterization space.
 
Note that the assumption \eqref{HBlowUpRate} is to control the growth rate of the mean curvature which when combined with \eqref{diamBound} gives a bound on the diameter of $U_T$ in Lemma \ref{DiameterEstimate}. In the rotationally symmetric case this assumption appears through the choice of $D$ which is the distance which defines the tubular neighborhood $T_D(\Sigma_0)$ around the symmetric sphere $\Sigma_0$ of area $A_0$. Hence the radial distance is also assumed to be bounded in the rotationally symmetric case.

We also point out that \eqref{HUpperBound} is used to obtain the $C^0$ convergence from below assumption in Lemma \ref{C0BoundBelow} which was observed to be so important in the work of the author and Sormani \cite{BS}, in the case of warped products, in order to show that $L^2$ convergence agreed with the GH and SWIF convergence. This is related to the author's previous work where $L^2$ stability was obtained.

Lastly, \eqref{Convexity} and \eqref{distanceAssumption1} are used to control a metric approximation quantity in Lemma \ref{DistancePreservingEstimate} which appears in the work of Lakzian and Sormani \cite{LaS}, and Lakzian \cite{L} where the authors were concerned with using smooth convergence away from a singular set in order to conclude SWIF convergence on a larger set. It is important to note that \eqref{distanceAssumption1} is comparing the distances on $\Sigma_0^i$ to a sphere with the same area radius (See Proposition \ref{avgH} for details on the parameterization). There is freedom in the argument to choose any area preserving diffeomorphism of $\Sigma_0$ and $S^2(r_0)$ so implied in this condition is that a choice of area preserving diffeomorphisms is made where \eqref{distanceAssumption1} is satisfied. We discuss condition \eqref{distanceAssumption1} in some detail in Example \ref{NoGHExample}. We note that \eqref{Convexity} is just restricting the kinds of foliations by IMCF we are allowed to consider when the IMCF becomes singular and is satisfied everywhere in the rotationally symmetric case.
\end{rmrk}

\begin{rmrk}\label{WeakSolutionsIMCF}
In the author's previous work \cite{BA2} it was noted that if $\Sigma_0$ is a minimizing hull then the main result of \cite{BA2}, Theorem \ref{SPMTPrevious}, also applies to the regions between jumps of the weak formulation of Huisken and Ilmanen if we stay away from the jump times. The important part of Theorem \ref{SPMTJumpRegionSmooth} and Theorem \ref{SPMTJumpRegion}  is that they allow $t=0$ and $t=T$ to be singular times of the weak solution to IMCF, i.e. $\Sigma_0$ could be the result of a jump of a weak solution of IMCF and $\Sigma_T$ could be a surface which will jump instantly to the outward minimizing hull under the weak solution to IMCF.

In order to see this it is important to remember three important lemmas of Huisken and Ilmanen:

\begin{itemize}

\item Smooth flows satisfy the weak formulation in the domain they foliate (Lemma 2.3 \cite{HI}).

\item The weak evolution of a smooth, $H > 0$, strictly minimizing hull is smooth for a short time (Lemma 2.4 \cite{HI}).

\item It can be shown that the weak solution remains smooth until the first moment when either $\Sigma_t \not= \Sigma_t'$, $H \searrow 0$ or $|A| \nearrow \infty$ where $\Sigma_t'$ is the outward minimizing hull of $\Sigma_t$ (Remark after Lemma 2.4 \cite{HI}). 
\end{itemize}

This illustrates why it is important to allow the bounds on $H$ and $|A|$ to degenerate as the solution approaches $t=0,T$. Also, if $\Sigma_0$ is the result of a jump of a weak solution of IMCF then it is expected for $H=0$ on some portion of $\Sigma_0$ and hence it is important to allow $H=0$ on $\Sigma_0^i$ and $\Sigma_T^i$. See Example \ref{Ex WeakConvergence} for a discussion of how the results of this paper can be used in combination with weak solutions of IMCF. 
\end{rmrk}

In Section \ref{sec-U-GH-FLATBackground}, we review the definitions and important theorems for Uniform, GH and SWIF convergence. In particular, we review the compactness theorem of Huang, Lee and Sormani \cite{LS1} which is used to obtain Uniform, GH and SWIF convergence to an unspecified length space under Lipschitz bounds. We also review the work on distance preserving maps of Lee and Sormani \cite{LS1} as well as the work of Lakzian and Sormani \cite{LaS} and Lakzian \cite{L} on smooth convergence away from singular sets. To this end we prove a similar theorem to \cite{LaS}, Theorem \ref{SWIFonCompactSets}, which is useful for our context since we can use the foliation by IMCF to achieve important metric estimates in section \ref{sect-Hawk}.

In Section \ref{sect-Hawk}, we review results from the author's previous work \cite{BA2} where $L^2$ convergence was obtained under less assumptions than Theorem \ref{SPMT} and Theorem \ref{SPMTJumpRegion}. We also obtain new results on the geodesic structure of a region of a manifold foliated by IMCF as well as metric estimates which are crucial for applying the results of section \ref{sec-U-GH-FLATBackground} in section \ref{sect-IMCF}.

In Section \ref{sect-IMCF}, we will use IMCF to show how to find Uniform, GH, and Flat convergence of $\hat{g}^j$ to $\delta$. In this section we establish all of the assumptions we need to show GH convergence which culminates in Theorem \ref{IMCFConv}. Also, interesting results on geodesics on regions foliated by IMCF are established in Lemma \ref{IMCFChristoffel}, Corollary \ref{IMCFgeod} and Corollary \ref{IMCFgeodCor} which we believe could be of independent interest.

In Section \ref{sect-Stability}, we show how to prove Theorem \ref{SPMT}  using the results of the last section. In this section we see how some of the assumptions that were needed for GH convergence in the last section follow from IMCF and the stability assumptions on the Hawking mass.

In Section \ref{sec Example}, we give three examples which serve to illustrate some of the hypotheses of Theorem \ref{SPMTJumpRegion} as well as discuss the application of the author's stability theorems in combination with a further understanding of properties of weak  and strong solutions of IMCF.

\vspace{.2in}
\noindent{\bf Acknowledgements:} I would like to thank the organizers, Piotr Chrusciel, Richard Schoen, Christina Sormani, Mu-Tao Wang, and Shing-Tung Yau, for the opportunity to speak on part of this work at the Simons Center for Geometry and Physics, ``Mass in General Relativity workshop." In particular I would like to thank Christina Sormani for her constant support.

\section{Background on Uniform, GH, and SWIF Convergence}\label{sec-U-GH-FLATBackground}

In this section we review the definitions of important notions of convergence for metric spaces as well as review important prior results related to these notions that will be used in this paper. Our aim is to give a brief introduction to these concepts without technical details so we will reference sources where the reader can obtain a more complete understanding if desired.

\subsection{Uniform Convergence}\label{subsec Uniform Def}

Consider the metric spaces $(X,d_1)$, $(X,d_2)$ and define the uniform distance between them to be
\begin{align}
d_{unif}(d_1,d_2) = \sup_{x,y\in X} |d_1(x,y) - d_2(x,y)|.
\end{align}
Notice that if you think of the metrics as functions, $d_i: X\times X \rightarrow \R$, then the uniform distance $d_{unif}(d_1,d_2)$ is equivalent to the $C^0$ distance between functions. We say that a sequence of metrics spaces $(X, d_j)$ converges to the metric space $(X,d_{\infty})$ if $d_{unif}(d_j, d_{\infty}) \rightarrow 0$ as $j \rightarrow \infty$.

One limitation of uniform convergence is that it requires the metric spaces to have the same topology. In our setting this is not a problem since it was shown in the author's previous result \cite{BA2} that $U_T^i$ eventually has the topology of $S^2\times \R$ and hence this notion of convergence shows up in Theorem \ref{SPMT}. See the text of Burago, Burago, and Ivanov  \cite{BBI} for more information on uniform convergence.

\subsection{Gromov-Hausdorff Convergence}\label{subsec GH Def}
Gromov-Hausdorff convergence was introduced by Gromov in \cite{G}
and which is discussed in the text of Burago, Burago, and Ivanov \cite{BBI}. It measures a distance
between metric spaces and is more general than uniform convergence since it doesn't require the metrics spaces to have the same topology.   It is an intrinsic version of the Hausdorff distance between
sets in a common metric space $Z$ which is defined as
\begin{equation} 
d_H^Z(U_1, U_2) = \inf\{ r \, : \, U_1\subset B_r(U_2) \textrm{ and } U_2\subset B_r(U_1)\},
\end{equation} 
where $B_r(U)=\{x \in Z: \, \exists y \in U \, s.t.\, d_Z(x,y)<r\}$.
In order to define the distance between a pair of compact metric spaces, $(X_i,d_i)$, which may
not lie in the same compact metric space, we use distance preserving maps to embed both metric spaces in a common, compact metric space $Z$. A distance preserving map is defined by
\begin{equation}
\varphi_i: X_i \to Z \textrm{ such that } d_Z(\varphi_i(p), \varphi_i(q)) = d_i(p,q) \,\, \forall p,q \in X_i
\end{equation}  
where it is important to note that we are requiring a metric isometry here which is stronger than a Riemannian isometry.

The Gromov-Hausdorff distance between two compact metric spaces, $(X_i,d_i)$,
is then defined to be
\begin{equation}
d_{GH}((X_1,d_1),(X_2,d_2))=\inf\{ d_H^Z(\varphi_1(X_1), \varphi_2(X_2)) \, : \, \varphi_i: X_i\to Z\}
\end{equation}
where the infimum is taken over all compact metric spaces $Z$ and all
distance preserving maps, $\varphi_i: X_i \to Z$.

\subsection{Sormani-Wenger Intrinsic Flat Convergence}\label{subsec SWIF Def}

Gromov-Hausdorff distance between metric spaces is an extremely powerful and useful notion of distance but has been observed to be poorly suited for questions involving scalar curvature \cite{S}. To this end we now define another notion of convergence, introduced by Sormani and Wenger in \cite{SW}, which is defined on integral currents spaces.

The idea is to build an intrisic version of the Flat distance on $\R^n$ of Federrer and Fleming \cite{FF} for any metric space. If this is to be succesful one needs a notion of integral currents on metric spaces and such a current structure was introduced by Ambrosio and Kirchheim \cite{AK} which is called an integral current space. The construction of Ambrosio and Kirchheim \cite{AK} allows one to define the flat distance for currents $T_1, T_2$ of an integral current space $Z$ as follows
\begin{align}
d_F^Z(X_1,X_2) = \inf\{\mass^n(A)+\mass^{n+1}(B): A + \partial B = X_1 - X_2\}.
\end{align}

Then Sormani and Wenger \cite{SW} used this notion of flat convergence to define the intrinsic notion of convergence for integral currents spaces $M_1=(X_1,d_1,T_1)$ and $M_2=(X_2,d_2,T_2)$ in an analogous way to GH convergence. The Sormani-Wenger Intrinsic Flat (SWIF) distance is defined as
\begin{align}
d_{\mathcal{F}}(M_1,M_2) = \inf\{d_F^Z(\varphi_{1\#}T_1,\varphi_{2\#}T_2: \varphi_j X_j \rightarrow Z\}
\end{align} 
where the infimum is taken over all complete metric spaces $Z$ and all metric isometric embeddings $\varphi_j: X_j \rightarrow Z$ such that
\begin{align}
d_Z(\varphi_j(x),\varphi_j(y)) = d_{X_j}(x,y) \text{ } \forall x,y \in X_j.
\end{align}

In \cite{SW} Sormani and Wenger prove many important properties about this notion of convergence and since much work has been done to obtain a further understanding of this notion of convergence, some of which will be reviewed below. See Sormani \cite{S} for many interesting examples of SWIF convergence and its relationship to GH convergence.

\subsection{Estimating Distance Preserving Maps} \label{subsec Estimating Dist Maps}

It is important to note that a Riemannian isometry between Riemannian manifolds is weaker than a distance preserving map between the corresponding metric spaces. Often times when attempting to estimate the SWIF distance between Riemannian manifolds one would like to deduce information about a distance preserving map from a Riemannian isometry and one tool in this direction is provided by the following theorem of Lee and Sormani if you can estimate the quantity \eqref{EmbeddingConstant}. 

\begin{thm}[Lee and Sormani \cite{LS1}]\label{EmbeddingConstantThm}
Let $\varphi : M \rightarrow N$ be a Riemannian isometric embedding and let
\begin{equation}
C_M:= \sup_{p,q \in M} \left(d_M(p,q) - d_N(\varphi(p),\varphi(q))\right).\label{EmbeddingConstant}
\end{equation}
If
\begin{align}
Z=\{(x,0): x \in N\} \cup \{(x,s):x \in \varphi(M),s \in [0,S_M]\} \subset N \times [0,S_M]
\end{align}
where 
\begin{align}
S_M=\sqrt{C_M(diam(M)+C_M)}.\label{SeparationDistance}
\end{align}
Then $\psi: M \rightarrow Z$ defined as $\psi(x)=(\varphi(x),S_m)$, is an isometric embedding into $(Z,d_Z)$ where $d_Z$ is the induced length metric from the isometric product metric on $N\times[0,S_M]$.
\end{thm}

We will use this result in subsection \ref{subsec SWIF on Exhaustion} where the constant \eqref{EmbeddingConstant} appears in \eqref{EmbeddingConstantToZero}.

\subsection{Estimating SWIF Distance} \label{subsec Estimating SWIF Dist}

With Theorem \ref{EmbeddingConstantThm} in mind Lee and Sormani were able to give an important estimate of the SWIF distance between Riemannian manifolds.

\begin{thm}[Lee and Sormani \cite{LS1}]\label{EstimatingSWIFDistance}
If $\varphi_i: M_i^n \rightarrow N^{n+1}$ are Riemannian isometric embeddings with embedding constants $C_{M_i}$ as in \eqref{EmbeddingConstant}, and if they are disjoint and lie in the boundary of a region $W \subset N$ then
\begin{align}
d_{\mathcal{F}}(M_1,M_2) &\le S_{M_1}(Vol_n(M_1) + Vol_{n-1}(\partial M_1))
\\&+ S_{M_2}(Vol(M_2) + Vol_{n-1}(\partial M_2))
\\&+Vol_{n+1}(W) + Vol_n(V)
\end{align}
where $V = \partial W \setminus(\varphi(M_1)\cup \varphi(M_2))$ and $S_{M_i}$ are defined in \eqref{SeparationDistance}.
\end{thm}

We will use this result in the proof of Theorem \ref{SWIFonCompactSets}.

\subsection{Key Compactness Theorem} \label{subsec Compactness Thm}

An important compactness theorem was introduced by Wenger \cite{W}, Wenger's Compactness Theorem, which says that given an integral current space $M=(X,d,T)$ where
\begin{align}
Vol(M) \le V_0 \hspace{1cm} Vol(\partial M) \le A_0 \hspace{1cm} Diam(M) \le M_0
\end{align}
then a subsequence exists which converges in the SWIF sense to an integral current space which could be the zero space. This compactness theorem is important to understanding SWIF convergence and can be very useful in applications. In the case of Theorem \ref{SPMT} we expect to get uniform, GH and SWIF convergence and hence it is advantageous to have a compactness theorem which guarantees that all three of these notions of convergence agree. In the paper by Huang, Lee and Sormani \cite{HLS} on stability of the PMT for graphs a similar compactness theorem was needed, and hence proven, which we review here.

\begin{thm}[Huang, Lee and Sormani \cite{HLS}]\label{HLS-thm}
Fix a precompact $n$-dimensional integral current space $(X, d_0, T)$
without boundary (e.g. $\partial T=0$) and fix
$\lambda>0$.   Suppose that
$d_j$ are metrics on $X$ such that
\begin{align}\label{d_j}
\lambda \ge \frac{d_j(p,q)}{d_0(p,q)} \ge \frac{1}{\lambda}.
\end{align}
Then there exists a subsequence, also denoted $d_j$,
and a length metric $d_\infty$ satisfying (\ref{d_j}) such that
$d_j$ converges uniformly to $d_\infty$
\begin{align}\label{epsj}
\epsilon_j= \sup\left\{|d_j(p,q)-d_\infty(p,q)|:\,\, p,q\in X\right\} \to 0.
\end{align}
Furthermore
\begin{align}\label{GHjlim}
\lim_{j\to \infty} d_{GH}\left((X, d_j), (X, d_\infty)\right) =0
\end{align}
and
\begin{align}\label{Fjlim}
\lim_{j\to \infty} d_{\mathcal{F}}\left((X, d_j,T), (X, d_\infty,T)\right) =0.
\end{align}
In particular, $(X, d_\infty, T)$ is an integral current space
and $\set(T)=X$ so there are no disappearing sequences of
points $x_j\in (X, d_j)$.

In fact we have
\begin{align}\label{GHj}
d_{GH}\left((X, d_j), (X, d_\infty)\right) \le 2\epsilon_j
\end{align}
and 
\begin{align}\label{Fj}
d_{\mathcal{F}}\left((X, d_j, T), (X, d_\infty, T)\right) \le
2^{(n+1)/2} \lambda^{n+1} 2\epsilon_j \mass_{(X,d_0)}(T).
\end{align}
\end{thm}

We will specifically use this theorem in the proof of Theorem \ref{IMCFConv} to get a subsequence which converges to a length metric. Then in order to identify this length metric as Euclidean space we will show pointwise convergence of distances in Corollary \ref{limsupEst} and Corollary \ref{distLowerBound2}.

\subsection{Contrasting $L^2$, GH and SWIF Convergence} \label{subsec ContastingConvergence}
In applications where one expects SWIF convergence for a sequence of Riemannian manifolds it has been noticed that one often obtains $L^2$ convergence or $W^{1,2}$ convergence more immediately (See \cite{BA2,BAetal, Br}). This motivated the author and Christina Sormani to investigate the connections between $L^2$ convergence and SWIF convergence in \cite{BS} where we proved the following theorem for warped products.

\begin{thm}[BA and Sormani \cite{BS}]\label{WarpConv} 
  Consider the warped product manifolds $M^n=[r_0,r_1]\times_{f_i} \Sigma$, where $\Sigma$ is an $n-1$ dimensional manifold. Assume the warping factors, $f_j\in C^0(r_0,r_1)$ , satisfy the following:
  \begin{equation}
 0<f_{\infty}(r)-\frac{1}{j} \le f_j(r) \le K < \infty 
  \end{equation}
and
  \begin{equation}
  f_j(r) \rightarrow f_{\infty}(r)> 0 \textrm{ in }L^2
  \end{equation}
  where $f_\infty \in C^0(r_0,r_1)$.
  
  Then we have GH and $\mathcal{F}$ convergence of the 
  warped product manifolds,
  \begin{equation}
  M_j=[r_0,r_1]\times_{f_j} \Sigma \to   M_\infty=[r_0,r_1]\times_{f_\infty} \Sigma,
  \end{equation}
  and uniform convergence of their distance functions, $d_ j \to d_\infty$.
 \end{thm}
 
 This theorem strictly speaking does not apply to this setting since we are not dealing with warped products but it should be noted that the insight gained by working on this paper has informed many of the proofs in section \ref{sect-IMCF}. In the beginning of section \ref{sect-IMCF} we give a discussion of how Theorem \ref{WarpConv} is related to the proof of Theorem \ref{SPMT}.

\subsection{SWIF Convergence on Exhaustion of Sets} \label{subsec SWIF on Exhaustion}

One important way to estimate the SWIF distance between Riemannian manifolds is when one has smooth convergence away from a singular set which was developed by Lakzian and Sormani \cite{LaS} and Lakzian \cite{L}. Lakzian and Sormani give conditions which if satisfied in conjunction with smooth convergence away from a singular set imply SWIF convergence. 

In the case of Theorem \ref{SPMTJumpRegion} we do not want to assume smooth convergence on the precompact exhaustion but rather we would like to use Theorem \ref{SPMT} and hence we are free to assume GH and SWIF convergence on the exhaustion. With this in mind we will have to add more assumptions to Theorem \ref{ConvergenceAwayFromSingularSets} in our case but we also have the added benefit of being able to leverage IMCF in order to satisfy these additional assumptions. With this in mind we state the following theorem which gives a way of using SWIF convergence on an exhaustion to conclude SWIF convergence on the larger set which will be used to prove Theorem \ref{SPMTJumpRegion}.

Let $\Sigma^n$ be a manifold and define the Riemannian manifolds
\begin{align}
M_i^k&= (\Sigma\times[t_1^k,t_2^k],g_i) \subset M_i=(\Sigma \times [0,T],g_i)
\\M_0^k &= (\Sigma\times[t_1^k,t_2^k],g_0) \subset M_0=(\Sigma \times [0,T],g_0)
\end{align}
where $0 < t_1^k < t_2^k <T < \infty$, $t_1^k \rightarrow 0, t_2^k \rightarrow T$ as $k \rightarrow \infty$. 

\begin{thm}\label{SWIFonCompactSets}
If 
\begin{align}
d_{\mathcal{F}}(M_i^k,M_0^k) \rightarrow 0 \text{ as } i \rightarrow \infty \text{ } \forall k \in \N,  
\end{align}
and
\begin{align}
Vol(\partial M_0^k)&\le A_0, \hspace{1cm} Vol(\partial M_i^k)\le A_1,
\\Vol(M_0)&= V_0, \hspace{1cm} Vol(M_i) \le V_1
\\ Diam(M_0) &= D_0, \hspace{1cm} Diam(M_i) \le D_1,
\\ \limsup_{i\rightarrow \infty}Vol(M_i\setminus M_i^k)& = \gamma_k \hspace{1cm} \lim_{k\rightarrow \infty}\gamma_k = 0
\\ \limsup_{i\rightarrow \infty}\sup_{p,q \in M_i^k}(d_{M_i^k}(p,q) - d_{M_i}(p,q))& = \beta_k \hspace{1cm} \lim_{k\rightarrow \infty}\beta_k = 0 \label{EmbeddingConstantToZero}
\end{align}
then 
\begin{align}
d_{\mathcal{F}}(M_i,M_0) \rightarrow 0.
\end{align}
\end{thm}

\begin{proof}
By the triangle inequality
\begin{align}
d_{\mathcal{F}}(M_i,M_0) \le d_{\mathcal{F}}(M_i,M_i^k)+d_{\mathcal{F}}(M_i^k,M_0^k)+d_{\mathcal{F}}(M_0^k,M_0)
\end{align}
where we have assumed that $d_{\mathcal{F}}(M_i^k,M_0^k) \rightarrow 0$ and hence we are left to estimate the other two terms. We note that it will be sufficient to make the argument for $M_i^k$ and $M_i$.

Consider the construction of Theorem \ref{EmbeddingConstantThm}
\begin{align}
Z_i^k &=\left (\{(x,0):x \in M_i\} \cup \{ (x,s): x \in M_i^k, s \in [0,S_i^k]\} \right)
\\&\subset M_i \times [0,S_i^k]
\end{align}
where 
\begin{align}
C_i^k &=\sup_{p,q \in M_i^k}(d_{M_i^k}(p,q) - d_{M_i}(p,q))
\\S_i^k &= \sqrt{C_i^k(Diam(M_i^k) + C_i^k)}.
\end{align}
 Now define the distance preserving maps of Theorem \ref{EmbeddingConstantThm} $\varphi_i^k:M_i^k \rightarrow Z_i^k$ and $\varphi_i:M_i\rightarrow Z_i^k$ where 
\begin{align}
\varphi_i^k(x,t) = (x,t,S_i^k) \hspace{1 cm} \varphi_i(x,t) = (x,t,0).
\end{align}
We can construct an $A_i^k,B_i^k \subset Z_i^k$, in a similar way to Lee and Sormani \cite{LS1} Theorem \ref{EstimatingSWIFDistance}, so that
\begin{align}
\varphi_i^k(M_i^k) - \varphi_i(M_i) = \partial B_i^k + A_i^k
\end{align}
by setting
\begin{align}
B_i^k &= \{ (x,s): x \in M_i^k, s \in [0,S_i^k]\}
\\A_i^k&= \{(x,0):x \in M_i\setminus M_i^k\} \cup \{ (x,s): x \in \partial M_i^k, s \in [0,S_i^k]\}.
\end{align}
From the construction above we see that
\begin{align}
d_{\mathcal{F}}(M_i^k,M_i) &\le d_F^{Z_i^k}(M_i^k,M_i)
\\ &\le Vol(B_i^k)+Vol(A_i^k) 
\\ &\le S_i^k Vol(M_i^k)) +  Vol(M_i\setminus M_i^k))+ S_i^k Vol(\partial M_i^k).
\end{align}

Putting this all together we see that
\begin{align}
d_{\mathcal{F}}(M_i,M_0) &\le d_{\mathcal{F}}(M_i,M_i^k)+d_{\mathcal{F}}(M_i^k,M_0^k)+d_{\mathcal{F}}(M_0^k,M_0)
\\&\le S_0^k(V_0+A_0) +  Vol(M_0\setminus M_0^k)) +d_{\mathcal{F}}(M_i^k,M_0^k) 
\\&+S_i^k(V_1+A_1) +  Vol(M_i\setminus M_i^k))
\end{align}
and hence if we let $i \rightarrow \infty$ we find
\begin{align}
\limsup_{i \rightarrow \infty} d_{\mathcal{F}}(M_i,M_0) \le S_0^k(V_0+A_0) +  Vol(M_0\setminus M_0^k)) +C \beta_k(V_1+A_1) +  \gamma_k
\end{align}
and so by letting $k \rightarrow \infty$ we find $d_{\mathcal{F}}(M_i,M_0) \rightarrow 0$.
\end{proof} 

\section{Estimates Using IMCF Coordinates} \label{sect-Hawk}

The main tool in this paper is IMCF which we remember is defined for surfaces $\Sigma^n \subset M^{n+1}$ evolving through a one parameter family of embeddings $F: \Sigma \times [0,T] \rightarrow M$, $F$ satisfying inverse mean curvature flow

\begin{equation}
\begin{cases}
\frac{\partial F}{\partial t}(p,t) = \frac{\nu(p,t)}{H(p,t)}  &\text{ for } (p,t) \in \Sigma \times [0,T)
\\ F(p,0) = \Sigma_0  &\text{ for } p \in \Sigma 
\end{cases}
\end{equation}
where $H$ is the mean curvature of $\Sigma_t := F_t(\Sigma)$ and $\nu$ is the outward pointing normal vector. The outward pointing normal vector will be well defined in our case since we have in mind, $M^3$, an asymptotically flat manifold with one end. For a glimpse of long time existence and asymptotic analysis results for smooth IMCF in various ambient manifolds and a discussion of the history of what is known see \cite{S1,S2}.

\subsection{Previous Results on IMCF and Stability} \label{subsect-PreviousIMCFStability}

In the author's previous work  on stability of the PMT under $L^2$ convergence \cite{BA2} many important consequences of using IMCF coordinates on a sequence of manifolds whose mass is going to zero was derived. Here we review the important results which will be used in this paper. For proofs of these results see the author's original paper \cite{BA2}.

One extremely important quantity when studying IMCF is the Hawking mass which is defined as

\begin{align}
m_H(\Sigma) = \sqrt{\frac{|\Sigma|}{(16\pi)^3}} \left (16 \pi - \int_{\Sigma} H^2 d \mu \right ).
\end{align}

This quantity was noticed by Geroch to be monotone under smooth IMCF and motivated the study of this geometric evolution equation. In this work we say that the mass of a region is going to zero if the Hawking mass of the outermost leaf of the foliation by IMCF is going to zero, i.e. $m_H(\Sigma_T) \rightarrow 0$. We start by noting some simple consequences of this assumption.

\begin{lem}[Lemma 3.1 of \cite{BA2}] \label{naiveEstimate}
Let $\Sigma^2 \subset M^3$ be a hypersurface and $\Sigma_t$ it's corresponding solution of IMCF. If $m_1 \le m_H(\Sigma_t) \le m_2$  then 
\begin{align}
|\Sigma_t|&=|\Sigma_0| e^t
\end{align}
where $|\Sigma_t|$ is the $n$-dimensional area of $\Sigma$.

In addition, if $m_{H}(\Sigma_T^i) \rightarrow 0$ then 
\begin{align}
\bar{H^2}_i(t):=\dashint_{\Sigma_t^i} H_i^2 d \mu \rightarrow  \frac{4}{r_0}e^{-t}\label{unifAvgHEst1}
\end{align}
for every $t\in [0,T]$.
\end{lem}

By rearranging the calculation of monotonicity of the Hawking mass under IMCF one can deduce the following important consequences for stability.

\begin{lem}[Corollary 2.4 of \cite{BA2}] \label{GoToZero}Let $\Sigma^i\subset M^i$ be a compact, connected surface with corresponding solution to IMCF $\Sigma_t^i$. If $m_H(\Sigma_0)\ge0$ and $m_H(\Sigma^i_T) \rightarrow 0$  then for almost every $t \in [0,T]$ we have that
\begin{align}
&\int_{\Sigma_t^i} \frac{|\nabla H_i|^2}{H_i^2}d \mu \rightarrow 0 \hspace{.7 cm} \int_{\Sigma_t^i} (\lambda_1^i-\lambda_2^i)^2d \mu \rightarrow 0\hspace{.5 cm} \int_{\Sigma_t^i} R^i d \mu \rightarrow 0
\\ &\int_{\Sigma_t^i} Rc^i(\nu,\nu)d \mu \rightarrow 0 \hspace{.4 cm} \int_{\Sigma_t^i} K_{12}^id \mu \rightarrow 0 \hspace{1.5 cm} \int_{\Sigma_t^i} H_i^2 d\mu\rightarrow 16\pi \label{RicciEstimate}
\\&\int_{\Sigma_t^i} |A|_i^2 d \mu \rightarrow 8 \pi\hspace{.9 cm} \int_{\Sigma_t^i} \lambda_1^i\lambda_2^i d \mu \rightarrow 4\pi \hspace{1.3 cm} \chi(\Sigma_t^i) \rightarrow 2
\end{align}
as $i \rightarrow \infty$ where $K_{12}$ is the ambient sectional curvature tangent to $\Sigma_t$. Since $\chi(\Sigma_t^i)$ is discrete we see by the last convergence that $\Sigma_t^i$ must eventually become topologically a sphere. 
\end{lem}

Using the assumption on the Isoperimetric constant of Definition \ref{IMCFClass} one can deduce $L^2$ control on $H$ using Lemma \ref{GoToZero}. In addition, we define special coordinates on $U_T$ by identifying it with an annulus in $\R^n$. This is done by mapping each $\Sigma_t$ to the corresponding sphere, centered at the origin in $\R^n$, with the same area radius via an area preserving diffeomorphism of $\Sigma_0$ which is then propagated to each $\Sigma_t$ by IMCF.

\begin{prop}[Proposition 2.9 of \cite{BA2}]\label{avgH}If $\Sigma_t^i$ is a sequence of IMCF solutions where $\int_{\Sigma_t^i} \frac{|\nabla H|^2}{H^2}d \mu \rightarrow 0$ as $i \rightarrow \infty$, $0 < H_0 \le H(x,t) \le H_1 < \infty$ and $|A|(x,t) \le A_0 < \infty$ then
\begin{align}
\int_{\Sigma_t^i} (H_i - \bar{H}_i)^2 d \mu \rightarrow 0
\end{align}
as $i \rightarrow \infty$ for almost every $t \in [0,T]$ where $\bar{H}_i = \dashint_{\Sigma_t^i}H_id \mu$. 

Let $d\mu_t^i$ be the volume form on $\Sigma$ w.r.t. $g^i(\cdot,t)$ then we can find a parameterization of $\Sigma_t$ so that 
\begin{align}
d\mu_t^i = r_0^2 e^t d\sigma
\end{align}
where $d\sigma$ is the standard volume form on the unit sphere.

Then for almost every $t \in [0,T]$ and almost every $x \in \Sigma$, with respect to $d\sigma$, we have that 
\begin{align}
H_i(x,t) - \bar{H}_i(t) \rightarrow  0,
\end{align}
along a subsequence. 
\end{prop}

Lemma \ref{GoToZero} gives a lot of important information but since we do not know the sign of $Rc(\nu,\nu)$ we do not get much information out of $\int_{\Sigma_t^i} Rc(\nu,\nu)d \mu \rightarrow 0$ and so a weak $L^2$ convergence result is important to being able to use the curvature assumptions of Theorem \ref{SPMT}.

\begin{lem}[Lemma 2.10 of \cite{BA2}]\label{WeakRicciEstimate}Let $\Sigma^i_0\subset M^3_i$ be a compact, connected surface with corresponding solution to IMCF $\Sigma_t^i$. Then if $\phi \in C_c^1(\Sigma\times (a,b))$  and $0\le a <b\le T$ we can compute the  estimate
\begin{align}
\int_a^b\int_{\Sigma_t^i}& 2\phi Rc^i(\nu,\nu)d\mu d t=  \int_{\Sigma_a^i} \phi H_i^2 d\mu -  \int_{\Sigma_b^i} \phi H_i^2 d\mu 
\\&+ \int_a^b\int_{\Sigma_t^i}2\phi\frac{|\nabla H_i|^2}{H_i^2}-2\frac{\hat{g}^j(\nabla \phi, \nabla H_i)}{H_i} +\phi(H_i^2-2|A|_i^2) d\mu 
\end{align}

If $m_H(\Sigma^i_T) \rightarrow 0$ and $\Sigma_t$ satisfies the hypotheses of Proposition \ref{avgH} then the estimate above implies 
\begin{align}
\int_a^b\int_{\Sigma_t^i} \phi Rc^i(\nu,\nu)d\mu dt \rightarrow 0
\end{align}
\end{lem}

From this we can obtain a weak convergence result for $K_{12}$, the ambient sectional curvature tangent to $\Sigma_t$.

\begin{Cor}
Let $\Sigma^i_0\subset M^3_i$ be a compact, connected surface with corresponding solution to IMCF $\Sigma_t^i$. Then if $\phi \in C_c^1(\Sigma\times (a,b))$, $0\le a <b\le T$, and  $m_H(\Sigma^i_T) \rightarrow 0$ then 
\begin{align}
\int_a^b\int_{\Sigma_t} \phi K_{12}^i d \mu dt \rightarrow 0.
\end{align}
\end{Cor}
\begin{proof}
We can write
\begin{align}
\int_a^b\int_{\Sigma_t} \phi K_{12}^i d \mu dt&=\int_a^b\int_{\Sigma_t}  \frac{1}{2}\phi R^i - \phi Rc^i(\nu,\nu)  d \mu dt
\end{align}
where we know the Ricci term goes to zero by Lemma \ref{WeakRicciEstimate}. Then we notice
\begin{align}
\int_a^b\int_{\Sigma_t}  \frac{1}{2}\phi R^id \mu dt \le \frac{1}{2}\left(\max_{\Sigma \times (a,b)} \phi \right )\int_a^b\int_{\Sigma_t}   R^id \mu dt \rightarrow 0
\end{align}
by applying Lemma \ref{GoToZero} which proves the desired result.
\end{proof}

We now remember the main theorem of \cite{BA2} which gives $L^2$ stability of the PMT when a region is foliated by a uniformly controlled IMCF. 

\begin{thm}[Theorem 1.2 of \cite{BA2}]\label{SPMTPrevious}
Let $U_{T}^i \subset M_i^3$ be a sequence s.t. $U_{T}^i\subset \mathcal{M}_{r_0,H_0,I_0}^{T,H_1,A_1}$ and $m_H(\Sigma_{T}^i) \rightarrow 0$ as $i \rightarrow \infty$.  If we assume one of the following conditions,
\begin{enumerate}
\item $\exists$  $I > 0$ so that $K_{12}^i \ge 0$  and diam$(\Sigma_0^i) \le D$ $\forall$ $i \ge I$, 
\item $\exists$  $[a,b]\subset [0,T]$ such that $\| Rc^i(\nu,\nu)\|_{W^{1,2}(\Sigma\times [a,b])} \le C$ and diam$(\Sigma_t^i) \le D$ $\forall$ $i$, $t\in [a,b]$,where $W^{1,2}(\Sigma\times [a,b])$ is defined with respect to $\delta$,
\end{enumerate}
then
\begin{align}
\hat{g}^i \rightarrow \delta
\end{align}
in $L^2$ with respect to $\delta$.
\end{thm}

Now we also review two important theorems which were used to prove Theorem \ref{SPMTPrevious} and which will be used in section \ref{sect-IMCF}, Lemma \ref{LineApprox}.

\begin{thm}[Theorem 3.1 of \cite{BA2}]\label{gtog1} Let $U_{T}^i \subset M_i^3$ be a sequence such that $U_{T}^i\subset \mathcal{M}_{r_0,H_0,I_0}^{T,H_1,A_1}$ and $m_H(\Sigma_{T}^i) \rightarrow 0$ as $i \rightarrow \infty$ or $m_H(\Sigma_{T}^i)- m_H(\Sigma_{0}^i) \rightarrow 0$ and $m_H(\Sigma_T^i)\rightarrow m > 0$. If we define the metrics 
\begin{align}
\hat{g}^i(x,t) &= \frac{1}{H_i(x,t)^2} dt^2 + g^i(x,t)
\\g^i_1(x,t)&= \frac{1}{\overline{H}_i(t)^2}dt^2 + g^i(x,t)
\end{align}
 on $U_T^i$ then we have that
\begin{align}
\int_{U_T^i}|\hat{g}^i -g^i_1|^2 dV \rightarrow 0 
\end{align}
where $dV$ is the volume form on $U_T^i$. 
\end{thm}

\begin{thm}[Theorem 3.3 of \cite{BA2}]\label{g2tog3} Let $U_{T}^i \subset M_i^3$ be a sequence s.t. $U_{T}^i\subset \mathcal{M}_{r_0,H_0,I_0}^{T,H_1,A_1}$ and $m_H(\Sigma_{T}^i) \rightarrow 0$ as $i \rightarrow \infty$. If we define the metrics \begin{align}
g^i_2(x,t)&= \frac{1}{\bar{H}^i(t)^2}dt^2 + e^tg^i(x,0)
\\g^i_3(x,t)&= \frac{r_0^2}{4}e^tdt^2 + e^tg^i(x,0)
\end{align}
 on $U_T^i$ then we have that
\begin{align}
\int_{U_T^i}|g^i_2 -g^i_3|^2 dV \rightarrow 0 
\end{align}
where $dV$ is the volume form on $U_T^i$.
\end{thm}

Lastly we remind the reader of a type of $L^2$ stability which was obtained by the author in \cite{BA2} which requires very minimal hypotheses. This theorem says that the manifold must be getting $L^2$ close to a warped product which is Euclidean space if $g^i(x,0) = \sigma$.

\begin{thm}\label{SPMTWarped}
Let $U_{T}^i \subset M_i^3$ be a sequence s.t. $U_{T}^i\subset \mathcal{M}_{r_0,H_0,I_0}^{T,H_1,A_1}$ and $m_H(\Sigma_{T}^i) \rightarrow 0$ as $i \rightarrow \infty$. 
If we define the metrics 
\begin{align}
\hat{g}^i(x,t) &= \frac{1}{H_i(x,t)^2} dt^2 + g^i(x,t)
\\g^i_3(x,t)&= \frac{r_0^2}{4}e^tdt^2 + e^tg^i(x,0)
\end{align}
 on $U_T^i$ then we have that
\begin{align}
\int_{U_T^i}|\hat{g}^i -g^i_3|^2 dV \rightarrow 0 
\end{align}
where $dV$ is the volume form on $U_T^i$. 
\end{thm}

This theorem is interesting since it only requires a smooth, uniformly controlled IMCF and $m_H(\Sigma_T^i) \rightarrow 0$ in order to show a type of $L^2$ stability. See Example \ref{Ex UsingWarpedProductStability} for a discussion of how this theorem can be used to obtain stronger stability.

\subsection{Geodesic Equations in IMCF Coordinates}\label{subsec IMCFGeodEqs}

Since we are using the foliation of a manifold by a solution to IMCF as special coordinates it is interesting to see what these coordinates can tell us about the length structure of the manifold which we investigate through the geodesic equation in this subsection.

Let $\gamma(s) = (T(s),\theta_1(s),...,\theta_n(s))=(T(s),\vec{\theta}(s))$ be a geodesic in $U_T$ then we are interested in the geodesic equations
\begin{align}
T''&= \Gamma_{00}^0 T'^2 + \Gamma_{i0}^0 T' \theta_i' + \Gamma_{ij}^0 \theta_i'\theta_j'\label{geodEq1}
\\ \theta_k''&= \Gamma_{ij}^k \theta_i'\theta_j' + \Gamma_{i0}^kT'\theta_i' + \Gamma_{00}^k T'^2 \label{geodEq2}
\end{align}
where $'$ represents derivatives with respect to $s$ and we use $i,j,k$ for directions tangent to $\Sigma_t$ and $0$ for the direction normal to $\Sigma_t$. 

Now we compute the Christoffel symbols for the metric $\hat{g}^i$ in the IMCF coordinates.
\begin{lem}\label{IMCFChristoffel} We can find the following expressions for the Christoffel symbols of $\hat{g}$ in terms of the IMCF coordinates
\begin{align}
\Gamma_{00}^0&=\frac{1}{2}g^{00} \left (g_{00,0} \right ) = \frac{1}{2} H^2 \left ( \frac{-2\partial_t H}{H^3}\right ) = \frac{-\partial_t H}{H}
\\ \Gamma_{i0}^0&=\frac{1}{2}g^{00} \left (g_{i0,0}+g_{00,i}-g_{i0,0} \right )=\frac{1}{2} H^2 \left (\frac{-2\partial_i H}{H^3}\right )= \frac{-\partial_i H}{H}
\\ \Gamma_{ij}^0&=\frac{1}{2}g^{00} \left (g_{i0,j}+g_{j0,i}-g_{ij,0}\right )=\frac{1}{2} H^2 \left (\frac{-2A_{ij}}{H} \right ) = -HA_{ij}
\\ \Gamma_{i0}^k&= \frac{1}{2}g^{kp}\left ( g_{ip,0}\right) = \frac{1}{2} \left ( \frac{2 A_{ip}}{H} \right ) = g^{kp} \frac{A_{ip}}{H}
\\ \Gamma_{00}^k&= \frac{1}{2} g^{kp} \left (-g_{00,p}\right ) = \frac{1}{2} g^{kp} \left ( -\frac{1}{H^2} \right )_p = g^{kp} \frac{\partial_pH}{H^3}
\end{align}
\end{lem}
\begin{proof}
These formulas follow from simple calculations using the formulas given in the statement of the Lemma.
\end{proof}

Now using Lemma \ref{IMCFChristoffel} we can rewrite \eqref{geodEq1} and \eqref{geodEq2}.

\begin{Cor}\label{IMCFgeod}In terms of the IMCF coordinates the geodesic equations of $\hat{g}$ can be written as
\begin{align}
T''&= \frac{-\partial_t H}{H} T'^2 + \frac{-\partial_i H}{H} T' \theta_i' - HA(\vec{\theta}',\vec{\theta}') \label{TgeodEQ}
\\ \theta_k''&=\Gamma_{ij}^k \theta_i'\theta_j' + g^{kj}\frac{A_{ij}}{H}T'\theta_i' + g^{kj}\frac{\partial_j H}{H^3} T'^2 \label{thetaGeodEQ}
\end{align}
\end{Cor}
\begin{proof}
Combine the expressions for the Christoffel symbols in terms of the IMCF coordinates of Lemma \ref{IMCFChristoffel} combined with \eqref{geodEq1} and \eqref{geodEq2}.
\end{proof}
Interestingly, we can use \eqref{TgeodEQ} and \eqref{thetaGeodEQ} to deduce some simple consequences about geodesics which are analogous to the warped product case.
\begin{Cor}\label{IMCFgeodCor} Let $\gamma(s) = (T(s),\theta_1(s),...,\theta_n(s))=(T(s),\vec{\theta}(s))$ be a geodesic then:
\begin{align}
&\text{ If }T'\equiv0 \text{ then } H A(\vec{\theta}',\vec{\theta}')\equiv 0\text{.}\label{TangenttoLeavesofFoliation}
\\&\text{If } \vec{\theta}'\equiv 0 \text{ then } \nabla^{\Sigma_{T(s)}}H \equiv 0 \text{ and hence } \Sigma_{T(0)} \label{NormaltoLeavesofFoliation}
\\&\text{ must have constant mean curvature along } \gamma.
\\&\text{If } T'(\bar{s}) = 0 \text{ then } T''(\bar{s}) =-HA(\vec{\theta},\vec{\theta})|_{\bar{s}}\text{.} \label{ConvexityMins}
\end{align}
\end{Cor}
\begin{proof}
If $T'\equiv 0$ then by \eqref{TgeodEQ} we find that $HA(\vec{\theta},\vec{\theta})\equiv 0$ and since we have assume that $H_0 \le H(x,t)$ the first conclusion follows.

If $\vec{\theta}'\equiv 0$ then by \eqref{thetaGeodEQ} we find that $0 \equiv g^{kj}\frac{\partial_j H}{H^3} T'^2$ and since in this case $T' \not = 0$ and $H_0 \le H(x,t)$ the second conclusion follows.
\end{proof}

\subsection{Metric Estimates Using IMCF} \label{MetricIMCFEstimates}

In order to be able to estimate the SWIF distance between Riemannian manifolds it is important to be able to control metric quantities such as the diameter and volume. We now move to prove several lemmas which give important control on metric quantities which will be used to prove Theorem \ref{SPMTJumpRegion} in Section \ref{sect-Stability}.

Throughout the rest of the paper we will use the slightly less cumbersome notation $U_T^{i,k} : =  U_{t_1^k}^{i,t_2^k}$.

\begin{lem}\label{VolumeEstimate} If $\Sigma_t$ is a solution to IMCF for $t \in [0,T]$ such that
\begin{align}
\int_{\Sigma_t} \frac{1}{H(x,t)} d \mu \le h(t) \text{ for }t \in [0,T]
\end{align}
 where $h \in L^1([0,T])$ then
\begin{align}
Vol(U_T) \le |\Sigma_0| e^T\int_0^T h(t) dt \le C.
\end{align}
In addition we find
\begin{align}
Vol(U_T\setminus U_T^k) \le |\Sigma_0| e^T\left(\int_0^{t_1^k} h(t)  dt + \int_{t_2^k}^T  h(t)  dt\right)
\end{align}
\end{lem}
\begin{proof}
We calculate
\begin{align}
Vol(M_i) &= \int_0^{T}\int_{\Sigma_t} \frac{1}{H_i(x,t)} d \mu dt \le \int_0^{T} |\Sigma_t|h(t)  dt  = |\Sigma_0|\int_0^{T} e^th(t)  dt 
\end{align}
and
\begin{align}
Vol(M_i\setminus M_i^k)) &= \int_0^{t_1^k}\int_{\Sigma_t} \frac{1}{H_i(x,t)} d \mu dt + \int_{t_2^k}^T\int_{\Sigma_t} \frac{1}{H_i(x,t)} d \mu dt
\\&\le \int_0^{t_1^k} |\Sigma_t|h(t)  dt + \int_{t_2^k}^T |\Sigma_t|h(t)  dt
\\&\le |\Sigma_0| \left( \int_0^{t_1^k} e^t h(t)  dt + \int_{t_2^k}^T e^th(t)  dt \right)
\end{align}
\end{proof}

\begin{lem}\label{DiameterEstimate}
If $\Sigma_t$ is a solution to IMCF for $t \in [0,T]$ such that $Diam(\Sigma_t) \le D$ and 
\begin{align}
 \frac{1}{H(x,t)}  \le h(t) \text{ for }t \in [0,T]
\end{align}
 where $h \in L^1([0,T])$ then
\begin{align}
Diam(U_T) \le  \int_0^T h(t) dt + D \le C.
\end{align}
\end{lem}
\begin{proof}
Let $p,q \in U_T$ so that $p \in \Sigma \times \{t_1\}$ and $q \in \Sigma\times\{t_2\}$ where $t_1,t_2 \in [0,T]$ and $t_1\le t_2$. Now define a curve $\gamma(s)$, parameterized by arc length on $\Sigma\times [0,T]$, which travels from $p$ to $q$ by first traveling solely in the $t$ direction from $p\in\Sigma_{t_1}$ to $q' \in \Sigma_{t_2}$ and then travels solely in the $\Sigma_{t_2}$ factor from $q'$ to $q$. Then we calculate
\begin{align}
d_{\hat{g}}(p,q) &\le L_{\hat{g}}(\gamma) = \int_{t_1}^{t_2}\frac{1}{H(\gamma(s))}ds + d_{\Sigma}(q',q)
\\& \le \int_0^T h(t) dt + D
\end{align}
\end{proof}

Now that we have obtained control on volume and diameter we move to control the uniformly well embedded property of $U_T^{i,k} \subset U_T^i$ which when combined with Lemma \ref{VolumeEstimate} and Lemma \ref{DiameterEstimate} will allow us to apply Theorem \ref{SWIFonCompactSets}.

\begin{lem}\label{DistancePreservingEstimate}
If $\Sigma_t^i$ is a solution to IMCF for $t \in [0,T]$ such that $Diam(\Sigma_t^i) \le D$,
\begin{align}
 \frac{1}{H_i(x,t)}  \le h(t) \text{ for }t \in [0,T]
\end{align}
 where $h \in L^1([0,T])$, 
 \begin{align}\label{distanceAssumption2}
 \liminf_{i \rightarrow \infty} d_{\Sigma_0^i}(\theta_1,\theta_2) \ge d_{S^2(r_0)}(\theta_1,\theta_2) \text{ } \forall  \theta_1,\theta_2 \in S^2, 
 \end{align} 
 \begin{align}
 \Sigma_{t_j^k} \rightarrow S^2(r_0e^{t_j^k}) \text{ in } C^0 \text{ for } j=1,2
 \end{align} 
uniformly in $i$ for each fixed $k \in \N$, and
 \begin{align} 
A^i(\cdot,\cdot) >0 \text{ on } U_T^i \setminus (U_T^{i,K}\cup \partial U_T^i)\label{convexity}
 \end{align}
 for some fixed $K \in \N$ then
\begin{align}
\lim_{k \rightarrow \infty}\limsup_{i \rightarrow \infty}\sup_{p,q \in U_T^{i,k}}(d_{U_T^{i,k}}(p,q) - d_{U_T^i}(p,q)) =0\label{WellEmbeddedToZero}.
\end{align}
\end{lem}

\begin{proof}
Consider $p,q \in  U_T^{i,k}$ and let the curve $C_i(s)=(T_i(s),\vec{\theta}_i(s))$ be the length minimizing curve between $p,q$ with respect to $\hat{g}^i$ in $U_T^i$. If we let $k \ge K$ then by the assumption that $A^i(\cdot,\cdot) >0$ on $U_T^i \setminus (U_T^{i,K}\cup \partial U_T^i)$ we know by Corollary \ref{IMCFgeodCor} \eqref{ConvexityMins} that $T_i$ can only have minimums and no maximums of $T$ in $s$ and hence a length minimizing curve between $p$ and $q$ can only have one minimum of $T$ in $s$. 

First, this tells us that if $C_i$ does not pass through $\Sigma_{t_1^k}$ then $C_i \subset U_T^{i,k}$ and hence there is no argument to make. If $C_i$ does pass through $\Sigma_{t_1^k}$ then this allows us to decompose $C_i$ as
\begin{align}
C_i = C_i^1 \cup C_i^2 \cup C_i^3
\end{align}
where $C_i^1$ is the piece of $C_i$ up to the first minimum in $T_i$, $C_i^2$ is the part of $C_i$ at the minimum in $T_i$ and $C_i^3$ is the part of $C_i$ after the only minimum in $T_i$. Note that $C_i^2$ is either a point in some $\Sigma_{t_i}$ or a length minimizing curve in $\partial U_T^i$ by \eqref{TangenttoLeavesofFoliation} and assumption \eqref{convexity}. We also note that $C_i^1$ and $C_i^2$ are monotone in $t$.

Notice that if $C_i^2$ is a length minimizing curve in $\partial U_T^i$ with endpoints $\theta_1, \theta_2$ such that 
\begin{align}
\liminf_{i \rightarrow \infty} d_{\Sigma_0^i}(\theta_1,\theta_2) > d_{S^2(r_0)}(\theta_1,\theta_2) \text{ } \forall  \theta_1,\theta_2 \in S^2,\label{DistStrictlyLarger}
\end{align}
then there exists a $\bar{t} \in (0,T)$ such that 
\begin{align}
\liminf_{i \rightarrow \infty} d_{\Sigma_0^i}(\theta_1,\theta_2) > d_{S^2(r_0e^{\bar{t}/2})}(\theta_1,\theta_2)> d_{S^2(r_0)}(\theta_1,\theta_2) \text{ }\forall  \theta_1,\theta_2 \in S^2.
\end{align}
This implies that for $i$ chosen large enough it would be shorter to connect $C_i^1$ and $C_i^3$ through a curve that lies in $\Sigma_{\bar{t}}^i$ because $\Sigma_{\bar{t}}^i \rightarrow S^2(r_0e^{\bar{t}/2})$ as $i \rightarrow \infty$. This would contradict the assumption that $C_i^2$ is a length minimizing curve in $\partial U_T^i$ and so we may assume that $C_i^2$ is a point which lies in $\Sigma_{t_i}$ or equality is achieved in \eqref{distanceAssumption2} with $C_i^2 \subset \Sigma_0$. 

\begin{figure}[H]\label{fig CurvesAppDistance}
\begin{tikzpicture}[scale=.5]
\draw (0,0) circle (1cm);
\draw[dashed] (0,0) circle (2cm);
\draw[dashed] (0,0) circle (5cm);
\draw (0,0) circle (6cm);
\draw (-2.2,4.2) node{p};
\draw[fill=black] (-1.98,3.98) circle (.08cm);
\draw (4.2,-2.2) node{q};
\draw[fill=black] (3.98,-1.98) circle (.08cm);
\draw (-2,4) .. controls (-.25,1.5) .. (.707,.707);
\draw (.707,.707) .. controls (1.5,-.25) .. (4,-2);
\draw (6.5,2.8) node{$U_T^i$};
\draw (-2.5,-2.5) node{$U_T^{i,k}$};
\draw (-2.1,2.6) node{$C_i^1$};
\draw (.2,.2) node{$C_i^2$};
\draw (1.8,1.8) node{$\bar{C}_i^2$};
\draw (2.6,-2.1) node{$C_i^3$};
\draw (0,-1.5) node{$\Sigma_0$};
\draw (0,-2.6) node{$\Sigma_{t_1^k}$};
\draw (0,-5.5) node{$\Sigma_{t_2^k}$};
\draw (0,-6.6) node{$\Sigma_T$};
\draw[line width=1mm] (1.93,-.55) arc(-15:105:2);
\draw[line width=.7mm] (.86,.49) arc(30:60:1);
\draw[fill=black] (-.5,1.9) circle (.08cm);
\draw[fill=black] (.5,.82) circle (.08cm);
\draw[fill=black] (.86,.49) circle (.08cm);
\draw[fill=black] (1.92,-.52) circle (.08cm);
\end{tikzpicture}
\caption{Curves approximating the distance between $p,q \in U_T^{i,k}$.}
\end{figure}
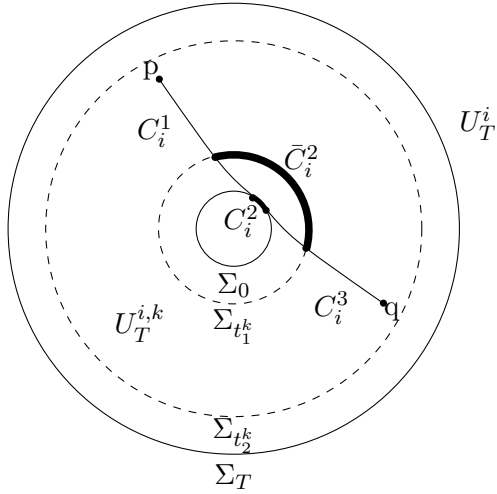

Now we define the curve $\bar{C}_i(s)$ between $p$ and $q$ inside of $U_T^{i,k}$ as
\begin{align}
\bar{C}_i = C_i^1|_{U_T^{i,k}} \cup \bar{C}_i^2 \cup C_i^3|_{U_T^{i,k}}
\end{align}
where $\bar{C}_i^2$ is a curve contained in $\partial U_T^{i,k}$, connecting the endpoints of $C_i^1|_{U_T^{i,k}}$ and $C_i^3|_{U_T^{i,k}}$, to be specified later. Then we note that 
\begin{align}
d_{U_T^i}(p,q) &= L_{\hat{g}^i}(C_i) 
\\ d_{U_T^{i,k}}(p,q) &\le L_{\hat{g}^i}(\bar{C}_i)
\end{align}
and hence
\begin{align}
0 &\le d_{U_T^{i,k}}(p,q) - d_{U_T^i}(p,q) \le L_{\hat{g}^i}(\bar{C}_i)-L_{\hat{g}^i}(C_i) 
\\& =-L_{\hat{g}^i}(C_i^1|_{U_T \setminus U_T^{i,k}})+ L_{g_i(\cdot,t_i^k)}(\bar{C}_i^2) - L_{g(\cdot,0)}(C_i^2)- L_{\hat{g}^i}(C_i^3|_{U_T \setminus U_T^{i,k}}).
\end{align}

Let $r_{i,k}^1,\bar{l}_{i,k}^1 \in U_T^i$ be the endpoints of $C_i^1|_{U_T \setminus U_T^{i,k}}$ where $r_{i,k}^1 \in \partial U_T^{i,k}$ and $\bar{l}_{i,k} \in \partial \Sigma_{t_i}$ (See figure \ref{fig ManyPointsDifferentCurves}). Then if we let $\bar{m}_{i,k}^1 \in \Sigma_{t_i}$ be the point which would be reached from $r_{i,k}^1$ if one traveled purely in the $t$ direction then we notice
\begin{align}
d_{\hat{g}^i}(r_{i,k}^1,\bar{l}_{i,k}^1) &\ge d_{\hat{g}^i}(\bar{m}_{i,k}^1,\bar{l}_{i,k}^1) - d_{\hat{g}^i}(r_{i,k}^1,\bar{m}_{i,k}^1)    
\end{align}
Now if we let $m_{i,k}^1, l_{i,k}^1  \in \Sigma_0$ be the points which would be reached from $\bar{m}_{i,k}^1, \bar{l}_{i,k}^1$, respectively, if one traveled purely in the $t$ direction then we notice
\begin{align}
d_{\hat{g}^i}(\bar{m}_{i,k}^1,l_{i,k}^1) &\ge d_{\hat{g}^i}(m_{i,k}^1,l_{i,k}^1) - d_{\hat{g}^i}(\bar{m}_{i,k}^1, m_{i,k}^1 )  - d_{\hat{g}^i}(\bar{l}_{i,k}^1,l_{i,k}^1)  
\\&=  d_{g^i(\cdot,0)}(m_{i,k}^1,l_{i,k}^1) - d_{\hat{g}^i}(\bar{m}_{i,k}^1, m_{i,k}^1 )  - d_{\hat{g}^i}(\bar{l}_{i,k}^1,l_{i,k}^1)  \label{distanceLowerBoundTriangleInequality}
\end{align}
where we remind the reader that $g^i(\cdot,t)$ is the induced metric on $\Sigma_t^i$ from $U_T^i$ and hence the last equality follows from the convexity assumption since $m_{i,k}^1, l_{i,k}^1 \in \partial U_T^i$. If we repeat this procedure for $C_i^3|_{U_T \setminus U_T^{i,k}}$ we will obtain points $r_{i,k}^3, l_{i,k}^3, \bar{l}_{i,k}^3, m_{i,k}^3, \bar{m}_{i,k}^3$ with a similar estimate. Note that $ l_{i,k}^j, m_{i,k}^j$ could be the same as $\bar{l}_{i,k}^j, \bar{m}_{i,k}^j$, $j=1,3$, respectively, if $t_i=0$.

\begin{figure}[H]\label{fig ManyPointsDifferentCurves}
\begin{tikzpicture}[scale=1]
\draw[line width=.3mm] (.7072,.7072) arc (45:135:1);
\draw[dashed] (2*.7072,2*.7072) arc (45:135:2);
\draw[dashed] (3*.7072,3*.7072) arc (45:135:3);
\draw[fill=black] (.7072,.7072) circle (.08cm);
\draw[fill=black] (-.7072,.7072) circle (.08cm);
\draw[fill=black] (2*.7072,2*.7072) circle (.08cm);
\draw[fill=black] (2*-.7072,2*.7072) circle (.08cm);
\draw[fill=black] (3*-.7072,3*.7072) circle (.08cm);
\draw (-2.3,1.8) node{$r_{i,k}^1$};
\draw (-.8,0) node{$m_{i,k}^1$};
\draw (.8,0) node{$l_{i,k}^1$};
\draw (-1.8,1) node{$\bar{m}_{i,k}^1$};
\draw (1.9,1) node{$\bar{l}_{i,k}^1$};
\draw[line width=.3mm] (-.7072,.7072) -- (3*-.7072,3*.7072);
\draw[line width=.3mm] (.7072,.7072) -- (2*.7072,2*.7072);
\draw[dashed] (2*.7072,2*.7072) -- (3*.7072,3*.7072);
\draw (3*-.7072,3*.7072) .. controls (.5,2.5) .. (2*.7072,2*.7072);
\draw (.25,.5) node{$\Sigma_0$};
\draw (.9,1.4) node{$\Sigma_{t_i}$};
\draw (1.5,2.2) node{$\Sigma_{t_1^k}$};
\end{tikzpicture}
\caption{Points and curves used to approximate $d_{\hat{g}^i}(r_{i,k}^1,\bar{l}_{i,k}^1)$.}
\end{figure}
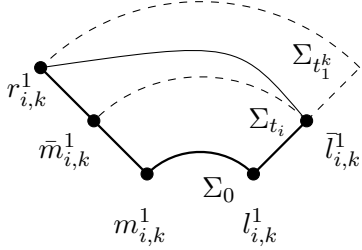

Notice that if $C_i^2$ is a point and \eqref{DistStrictlyLarger} holds then again for $i$ chosen large enough
\begin{align}
d_{g^i(\cdot,0)}(m_{i,k}^1,l_{i,k}^1) \ge d_{S^2(r_0)}(m_{i,k}^1,l_{i,k}^1)\label{DealingWithAPoint1}
\end{align}
which will only make the argument in \eqref{FirstDifference}, \eqref{SecondDifference}, and \eqref{ThirdDifference} easier and so we proceed with the more difficult case and make a note of where \eqref{DealingWithAPoint1}  factors into the argument later.

If equality is achieved in \eqref{distanceAssumption2} then for any subsequence where \eqref{DistStrictlyLarger} holds we can make the same argument as above to handle this case. So we may assume for the remainder of the argument that we are dealing with a susbequence such that
\begin{align}
\lim_{i \rightarrow \infty} d_{\Sigma_0^i}(\theta_1,\theta_2) = d_{S^2(r_0)}(\theta_1,\theta_2) \text{ } \forall  \theta_1,\theta_2 \in S^2.\label{DistEqual}
\end{align}
Then \eqref{distanceLowerBoundTriangleInequality} implies
\begin{align}
 0&\le d_{U_T^{i,k}}(p,q) - d_{U_T^i}(p,q) 
 \\&\le d_{\hat{g}^i}(\bar{m}_{i,k}^1, m_{i,k}^1 )  + d_{\hat{g}^i}(\bar{l}_{i,k}^1,l_{i,k}^1) + d_{\hat{g}^i}(r_{i,k}^1,m_{i,k}^1) - d_{g^i(\cdot,0)}(m_{i,k}^1,l_{i,k}^1)  
 \\&+ L_{g^i(\cdot,t_1^k)}(\bar{C}_i^2) - L_{g^i(\cdot,0)}(C_i^2)
 \\& d_{\hat{g}^i}(\bar{m}_{i,k}^3, m_{i,k}^3 )  + d_{\hat{g}^i}(\bar{l}_{i,k}^3,l_{i,k}^3) +d_{\hat{g}^i}(r_{i,k}^3,m_{i,k}^3) -  d_{g^i(\cdot,0)}(m_{i,k}^3,l_{i,k}^3)
 \\&\le d_{\hat{g}^i}(\bar{m}_{i,k}^1, m_{i,k}^1 )  + d_{\hat{g}^i}(\bar{l}_{i,k}^1,l_{i,k}^1)+d_{\hat{g}^i}(r_{i,k}^1,m_{i,k}^1)
 \\&+d_{\hat{g}^i}(r_{i,k}^3,m_{i,k}^3)+d_{\hat{g}^i}(\bar{m}_{i,k}^3, m_{i,k}^3 )  + d_{\hat{g}^i}(\bar{l}_{i,k}^3,l_{i,k}^3)
 \\&+L_{g^i(\cdot,t_1^k)}(\bar{C}_i^2) - L_{g^i(\cdot,0)}(C_i^2)-d_{g^i(\cdot,0)}(m_{i,k}^1,l_{i,k}^1)- d_{g^i(\cdot,0)}(m_{i,k}^3,l_{i,k}^3).
\end{align}
Notice that for every distance between points which only differ by a $t$ coordinate we can easily estimate
\begin{align}
d_{\hat{g}^i}(r_{i,k}^j,m_{i,k}^j) \le \int_0^{t_k} \frac{1}{H_i^2} dt \le \int_0^{t_k}h(t) dt
\end{align}
which is independent of $i$ and goes to $0$ as $k \rightarrow \infty$ and so now we move to estimate the remaining terms.

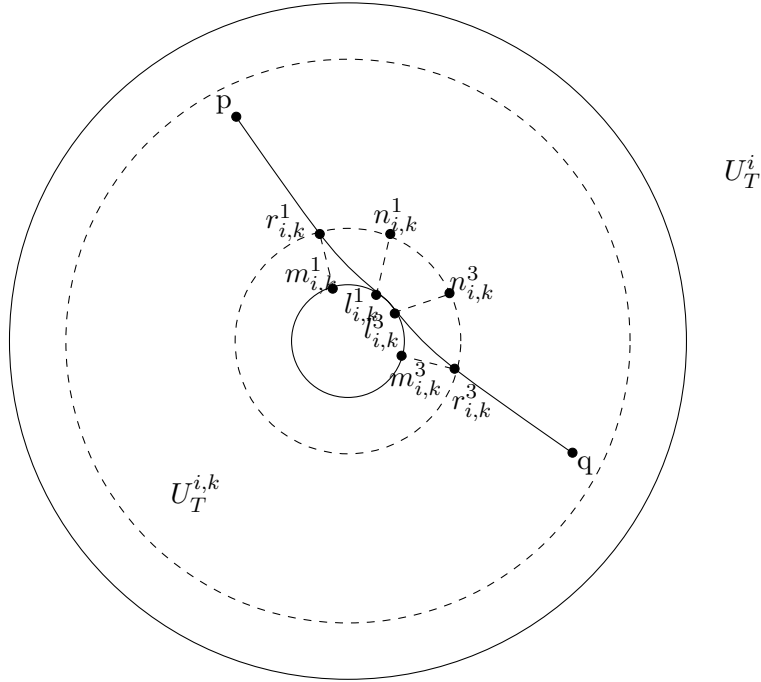
\begin{figure}[H]
\begin{tikzpicture}[scale=.75]
\draw (0,0) circle (1cm);
\draw[dashed] (0,0) circle (2cm);
\draw[dashed] (0,0) circle (5cm);
\draw (0,0) circle (6cm);
\draw (-2.2,4.2) node{p};
\draw[fill=black] (-1.98,3.98) circle (.08cm);
\draw (4.2,-2.2) node{q};
\draw[fill=black] (3.98,-1.98) circle (.08cm);
\draw (-2,4) .. controls (-.25,1.5) .. (.707,.707);
\draw (.707,.707) .. controls (1.5,-.25) .. (4,-2);
\draw (7,3) node{$U_T^i$};
\draw (-2.7,-2.7) node{$U_T^{i,k}$};
\draw[dashed] (.5,.8) -- (.75,1.9);
\draw[fill=black] (.5,.82) circle (.08cm);
\draw[fill=black] (.75,1.9) circle (.08cm);
\draw[dashed] (-0.25,.9) -- (-.5,1.9);
\draw[fill=black] (-.27,.93) circle (.08cm);
\draw[fill=black] (-.5,1.9) circle (.08cm);
\draw[dashed] (.9,-.25) -- (1.9,-.5);
\draw[fill=black] (.95,-.26) circle (.08cm);
\draw[fill=black] (1.89,-.49) circle (.08cm);
\draw[dashed] (.8,.5) -- (1.85,.85);
\draw[fill=black] (.83,.49) circle (.08cm);
\draw[fill=black] (1.8,.85) circle (.08cm);
\draw (-1.1,2.1) node{$r_{i,k}^1$};
\draw (-.65,1.15) node{$m_{i,k}^1$};
\draw (0.25,.6) node{$l_{i,k}^1$};
\draw (.85,2.2) node{$n_{i,k}^1$};
\draw (2.2,-1.1) node{$r_{i,k}^3$};
\draw (.6,.15) node{$l_{i,k}^3$};
\draw (1.2,-.7) node{$m_{i,k}^3$};
\draw (2.2,1) node{$n_{i,k}^3$};
\end{tikzpicture}
\caption{Geodesic between $p$ and $q$ with respect to $\hat{g}^i$ with approximating curves.}
\end{figure}

Let $n_{i,k}^j \in \partial U_T^{i,k}$, $j=1,3$, be the point which would be reached from $l_{i,k}^j$ if one traveled purely in the $t$ direction. Now we choose $\bar{C}_i^2$ to be a length minimizing geodesic with respect to $g_i(\cdot,t_i^k)$ connecting $r_{i,k}^1$, $n_{i,k}^1$, $n_{i,k}^3$, and $r_{i,k}^3$ (Note this implies that it is most likely not length minimizing from $r_{i,k}^1$ to $r_{i,k}^3$).

Now we estimate
\begin{align}
&L_{g^i(\cdot,t_1^k)}(\bar{C}_i^2) - L_{g_i(\cdot,0)}(C_i^2)-d_{g^i(\cdot,0)}(m_{i,k}^1,l_{i,k}^1)- d_{g^i(\cdot,0)}(m_{i,k}^3,l_{i,k}^3)\label{MainEqEnding}
\\&\le|d_{g^i(\cdot,t_1^k)}(r_{i,k}^1,n_{i,k}^1) -d_{g^i(\cdot,0)}(m_{i,k}^1,l_{i,k}^1)| 
\\&+|d_{g^i(\cdot,t_1^k)}(n_{i,k}^1,n_{i,k}^3) -d_{g^i(\cdot,0)}(l_{i,k}^1,l_{i,k}^3)|\label{Eq2EliminatedIfPoint}
\\&+|d_{g^i(\cdot,t_1^k)}(r_{i,k}^3,n_{i,k}^3) -d_{g^i(\cdot,0)}(m_{i,k}^3,l_{i,k}^3)|.
\end{align}
 Note that if $C_i^2$ is a point then $l_{i,k}^1=l_{i,k}^3$ and $n_{i,k}^1=n_{i,k}^3$ which eliminates \eqref{Eq2EliminatedIfPoint}. Notice that each of the three terms can be estimated in a similar way  and so we perform the estimate for only one of the terms as follows:
\begin{align}
&|d_{g^i(\cdot,t_1^k)}(r_{i,k}^1,n_{i,k}^1) -d_{g^i(\cdot,0)}(m_{i,k}^1,l_{i,k}^1)| \le
\\&|d_{r_0^2 e^{t_1^k}\sigma}(r_{i,k}^1,n_{i,k}^1) -d_{g^i(\cdot,t_1^k)}(r_{i,k}^1,n_{i,k}^1)| \label{FirstDifference}
\\+ & |d_{r_0^2e^{t_1^k}\sigma}(r_{i,k}^1,n_{i,k}^1) -d_{r_0^2 \sigma}(m_{i,k}^1,l_{i,k}^1)| \label{SecondDifference}
\\+& |d_{r_0^2\sigma}(m_{i,k}^1,l_{i,k}^1) -d_{g_i(\cdot,0)}(m_{i,k}^1,l_{i,k}^1)|. \label{ThirdDifference}
\end{align}
The \eqref{FirstDifference} and \eqref{ThirdDifference} go to $0$ as $i \rightarrow \infty$ by assumption (This is where \eqref{DealingWithAPoint1} is used to eliminate \eqref{ThirdDifference} if $C_i^2$ is a point and \eqref{DistStrictlyLarger} holds). Now we look at \eqref{SecondDifference} more closely and notice that since both metrics are on round spheres, $r_{i,k}^1$ to $l_{i,k}^1$ is purely in the $t$ direction, and $n_{i,k}^1$ to $m_{i,k}^1$ is purely in the $t$ direction then it is equivalent to just think of the metrics as living on the same parameterizing sphere. So if $m_{i,k}^1 = (t_0, \vec{\theta}_{i,k}^{1})$ and $l_{i,k}^1= (t_0, \vec{\theta}_{i,k}^{2})$ then we can write
\begin{align}
|d_{r_0^2e^{t_1^k}\sigma}(r_{i,k}^1,n_{i,k}^1) -d_{r_0^2 \sigma}(m_{i,k}^1,l_{i,k}^1)| = |d_{r_0^2e^{t_1^k}\sigma}( \vec{\theta}_{i,k}^{1}, \vec{\theta}_{i,k}^{2}) -d_{r_0^2 \sigma}( \vec{\theta}_{i,k}^{1}, \vec{\theta}_{i,k}^{2})|.
\end{align}

Now if we let $\vec{\theta}_N$ and $\vec{\theta}_S$ be the north and south pole on the sphere then we can estimate
\begin{align}
|d_{r_0^2e^{t_1^k}\sigma}(r_{i,k}^1,n_{i,k}^1) -d_{r_0^2 \sigma}(m_{i,k}^1,l_{i,k}^1)| \le |d_{r_0^2e^{t_1^k}\sigma}( \vec{\theta}_N, \vec{\theta}_S) -d_{r_0^2 \sigma}( \vec{\theta}_N,\vec{\theta}_S)|
\end{align}
which implies
\begin{align}
\lim_{k \rightarrow \infty} &\left ( \limsup_{i \rightarrow \infty}|d_{r_0^2e^{t_1^k}\sigma}(r_{i,k}^1,n_{i,k}^1) -d_{r_0^2 \sigma}(m_{i,k}^1,l_{i,k}^1)|\right)
\\& \le \lim_{k \rightarrow \infty} |d_{r_0^2e^{t_1^k}\sigma}( \vec{\theta}_N, \vec{\theta}_S) -d_{r_0^2 \sigma}( \vec{\theta}_N,\vec{\theta}_S)| = 0.
\end{align}
Putting this all together with \eqref{MainEqEnding} we achieve the desired estimate.
\end{proof}

\begin{rmrk}\label{RemovingConvexity}
It would be interesting to prove a similar result to Lemma \ref{DistancePreservingEstimate} without assumptions \eqref{distanceAssumption2} and/or \eqref{convexity}. A result of this kind would immediately imply a new stability result when combined with the results of this paper. It seems extremely difficult though to estimate \eqref{WellEmbeddedToZero} without these assumptions and hence the author is not even convinced that it is true.
\end{rmrk}

\subsection{Maximum Principle Estimates} \label{subsec-MaxPrincipleEstimates}

In this section we use lower bounds on Ricci curvature to obtain upper bounds on the mean curvature of a solution of IMCF via the maximum principle. These results will be used to obtain $C^0$ control on $\hat{g}^i$ from below which is an important step in proving Theorem \ref{SPMT}.

\begin{lem}
For $\Sigma_t$ a solution to IMCF so that $Rc(\nu,\nu) \ge -C$ for $C > 0$ then we find
\begin{align}
H(x,t) \le \sqrt{C_0 e^{-2t/n} + C n}
\end{align}
where $C_0 = \displaystyle \left (\max_{\Sigma_0}H\right)^2 - Cn$.
 \end{lem}
 \begin{proof}
 \begin{align}
 (\partial_t - \frac{1}{H} \Delta)H &= -\frac{2|\nabla H|^2}{H^3} - \frac{|A|^2}{H} - \frac{Rc(\nu,\nu)}{H}
 \\&\le -\frac{1}{n} H + \frac{C}{H} = \frac{1}{nH} \left(nC - H^2 \right)
 \end{align}
 This implies the following ODE inequality
 \begin{align}
 \frac{d}{dt} \left(\max_{\Sigma_t}H\right)^2 & \le \frac{1}{n}\left(nC - \left(\max_{\Sigma_t}H\right)^2\right)
 \end{align}
 and then the result follows by Hamilton's maximum principle applied to the evolution inequality given above.
 \end{proof}
 
 \begin{Cor}\label{C0BoundBelow}
 If $H_j(x,0)^2 \le \frac{4}{r_0^2} + \frac{C_1}{j}$ and $Rc^j(\nu,\nu) \ge -\frac{C_2}{j}$ for $C_1,C_2 > 0$ then we find
 \begin{align}
 \frac{1}{H_j(x,t)^2} \ge \frac{r_0^2}{4}e^t - \frac{C_3}{j}
 \end{align}
 for $C_3 > 0$.
 \end{Cor}
 \begin{proof}
  \begin{align}
 \frac{1}{H_j(x,t)^2} & \ge \frac{1}{\left [\displaystyle \left (\max_{\Sigma_0}H\right)^2 - \frac{nC_2}{j}\right ] e^{-2t/n} +\frac{n C_2}{j}}
 \\&\ge \frac{1}{\left [\frac{4}{r_0^2} + \frac{C_1}{j} - \frac{nC_2}{j}\right ] e^{-t} +\frac{n C_2}{j}}\ge \frac{r_0^2}{4}e^t - \frac{C_3}{j}
 \end{align}
 \end{proof}
 
 \section{Convergence of Manifolds Foliated by IMCF}\label{sect-IMCF}
 
 Let $\Sigma_t\subset M^3$ be a solution to IMCF starting at $\Sigma_0$ and consider $U_T = \{x \in \Sigma_t: t\in [0,T]\}$, the region in $M$ foliated by $\Sigma_t$. In this section we consider the metrics
 \begin{align}
 \hat{g}^j &= \frac{1}{H_j(x,t)^2} dt^2 + g^j(x,t)
 \\ \bar{g}^j &= \frac{r_0^2}{4}e^t dt^2 + g^j(x,t)
 \\ \delta &= \frac{r_0^2}{4}e^t dt^2 + + r_0^2 e^t \sigma(x)
 \end{align}
defined on the foliated region $U_T$. The goal is to show Uniform, GH, and SWIF convergence of $\hat{g}^j$ to $\delta$ by taking advantage of the intermediate metric $\bar{g}^j$. 

In this section we establish all the assumptions we need to show Uniform, GH and SWIF convergence which culminates in Theorem \ref{IMCFConv}. The ideas and consturctions used in this section were gleaned from working on Theorem \ref{WarpConv} with Christina Sormani in \cite{BS}. Though we do not work with warped products in this paper the author noted the relationship between $\hat{g}^i$ and warped products in Theorem \ref{SPMTWarped} and we use the intuition that each of the metrics defined above differs in at most one factor which was an important property used by the author and Sormani in \cite{BS}. In the next section we will see how some of these assumptions follow from assumptions on the Hawking mass and IMCF in order to prove Theorem \ref{SPMT}.

\subsection{Consequences of $L^2$ Assumptions}

We begin this subsection by estimating the difference in length measured by $\hat{g}^i$ and $\bar{g}^i$ for curves which are monotone in $t$. One should note the similarity of this Lemma \ref{IMCF-L} with Lemma 4.3 of \cite{BS}.

\begin{lem} \label{IMCF-L}
Fix a curve $C(s)=(T(s), \vec{\theta}(s))$ parameterized on $[0,1]$ which is monotone in $t$
then
\begin{align}
|L_{\hat{g}^j}(C)-L_{\bar{g}^i}(C)| \le \sqrt{T}\left(\int_{T(0)}^{T(1)} \left|\frac{1}{H_i(x,t)^2}-\frac{r_0^2}{4}\right|^2\,dt \right)^{1/4}.\label{lengthDistance}
\end{align}

\end{lem}
\begin{proof}
We start by computing the difference in the two lengths of $C$
\begin{eqnarray}
&&|L_{\hat{g}^j}(C)-L_{\bar{g}^i}(C)|
\\&=& \int_0^1| \sqrt{ \frac{T'(t)^2}{H_i^2} + g^i(\vec{\theta}',\vec{\theta}') }
- \sqrt{ \frac{T'(t)^2r_0^2}{4} + g^i(\vec{\theta}',\vec{\theta}') }|\, ds \\
&\le& \int_0^1 \sqrt{  \left|\frac{1}{H_i(x,t)^2}-\frac{r_0^2}{4}\right|} |T'(t)|\, ds \\
&\le&\left( \int_0^1 \sqrt{  \left|\frac{1}{H_i(x,t)^2}-\frac{r_0^2}{4}\right|}\,ds \right)^{1/2}
\left(\int_0^1 |T'(t)|^2\, ds \right)\label{Holder1}\\ 
&=&\left( \int_0^1  \left|\frac{1}{H_i(x,t)^2}-\frac{r_0^2}{4}\right|^2\,ds \right)^{1/4}
\left(\int_0^1 |T'(t)|^2\, ds \right)\label{Holder2}
\end{eqnarray}
where we use Holder's inequality in \eqref{Holder1} and \eqref{Holder2}.
If $C(s)=(T(s), \vec{\theta}(s))$ and $T'(s) > 0$ everywhere, then we can reparametrize 
so that $T(s)=t$ and hence
\begin{eqnarray}
&&|L_{\hat{g}^j}(C)-L_{g_1^i}(C)| \\
&\le& \left(\int_{T(0)}^{T(1)} \left|\frac{1}{H_i(x,t)^2}-\frac{r_0^2}{4}\right|^2\,dt \right)^{1/4}
\left(\int_{T(0)}^{T(1)} |T'(t)|^2\, dt\right)^{1/2} 
\end{eqnarray}
which also shows that $\left(\int_{T(0)}^{T(1)} |T'(t)|^2\, dt\right)^{1/2} \le \sqrt{T}$.
\end{proof}

One concern that we have with applying Lemma \ref{IMCF-L} is that \eqref{lengthDistance} may not go to zero for every curve $C$. In the next lemma we build a family of approximating curves where we know \eqref{lengthDistance} does go to zero.

\begin{lem}\label{LineApprox}
Let $C(s)=(T(s), \vec{\theta}(s))$ parameterized on $[0,1]$ be a straight line in the annulus $(\Sigma\times [0,T], \delta)$ so that $\left(\int_{T(0)}^{T(1)} \left|\frac{1}{H_i(\vec{\theta}(s),T(s))^2}-\frac{r_0^2}{4}\right|^2\,ds \right)^{1/4}$ does not converge to $0$. Then there is a sequence of curves 
\begin{align}
C_k(s)=(T_k(s), \vec{\theta}_k(s))=
\begin{cases}
\alpha(s) & s \in [-\epsilon,0)
\\ \gamma_{\epsilon}(s) & s \in [0,1]
\\\beta(s)& s \in (1,1+ \epsilon]
\end{cases}
\end{align}
 where $\gamma_{\epsilon}$ is a line parallel to $C$, $\alpha, \beta$ are curves connecting the endpoints of $\gamma_{\epsilon}$ and $C_k$, and $L_{\delta}(C_k) \rightarrow L_{\delta}(C)$ as $k \rightarrow \infty$. If we assume
 \begin{align}
 0 < H_0 \le H_i(x,t) \label{MeanCurvBounds}
 \end{align}
 then
\begin{align}
\int_{T_k(0)}^{T_k(1)} \left|\frac{1}{H_i(\vec{\theta}_k(s),T_k(s))^2}-\frac{r_0^2}{4}\right|^2\,dt \rightarrow 0
\end{align}
and hence
\begin{align}
|L_{\hat{g}^i}(C_k)-L_{\bar{g}}(C_k)| \le 2 \bar{C} \epsilon
\end{align}
as $i \rightarrow \infty$.
\end{lem}

\begin{figure}[H]\label{fig ApproximatingCurves}
\begin{tikzpicture}[scale=.7]
\draw (0,0) circle (1cm);
\draw (0,0) circle (6cm);
\draw (0,-3) node{$(U_T^i,\hat{g}^i)$};
\draw (0,-.6) node{$\Sigma_0$};
\draw (0,-5.6) node{$\Sigma_T$};
\draw (.8,5.5) node{$C(0)$};
\draw[fill=black] (0,5.5) circle (.08cm);
\draw (5.45,.7) node{$C(1)$};
\draw[fill=black] (5.5,0) circle (.08cm);
\draw (0,5.5) -- (5.5,0);
\draw (3.5,3.5) node{$C$};
\draw[fill=black] (-1.3,4.2) circle (.08cm);
\draw[fill=black] (4.2,-1.3) circle (.08cm);
\draw[dashed] (5.5,0) -- (4.2,-1.3);
\draw[dashed] (0,5.5) -- (-1.3,4.2);
\draw[dashed] (4.2,-1.3) -- (-1.3,4.2);
\draw (2,2) node{$C_k$};
\draw[->] (-.8,4.2) -- (0,5);
\draw[->] (4.2,-.8) -- (5,0);
\draw (1.1,1) node{$\gamma_{\epsilon}$};
\draw (-1.4,5.1) node{$\alpha$};
\draw (5.2,-1.3) node{$\beta$};
\draw (-1.5,-.7) node{$C(0)$};
\draw (-6.7,-.7) node{$C(1)$};
\draw (-1,0) -- (-6,0);
\draw[fill=black] (-1,0) circle (.08 cm);
\draw[fill=black] (-6,0) circle (.08 cm);
\draw[fill=black] (-5.9,1) circle (.08 cm);
\draw[dashed] (-5.9,1) -- (-1,1);
\draw[dashed] (-1,0) -- (-1,1);
\draw[fill=black] (-1,1) circle (.08cm);
\draw (-.5,.35) node{$\beta$};
\draw (-6.5,.5) node{$\alpha$};
\draw (-4,1.4) node{$\gamma_{\epsilon}$};
\draw (-4,.5) node{$C_k$};
\draw (-4,-.4) node{$C$};
\draw[->] (-5.8,.8) -- (-5.8,.1);
\draw[->] (-1.2,.8) -- (-1.2,.1);
\end{tikzpicture}
\caption{Two families of curves $C_k$ approximating $C$. }
\end{figure}
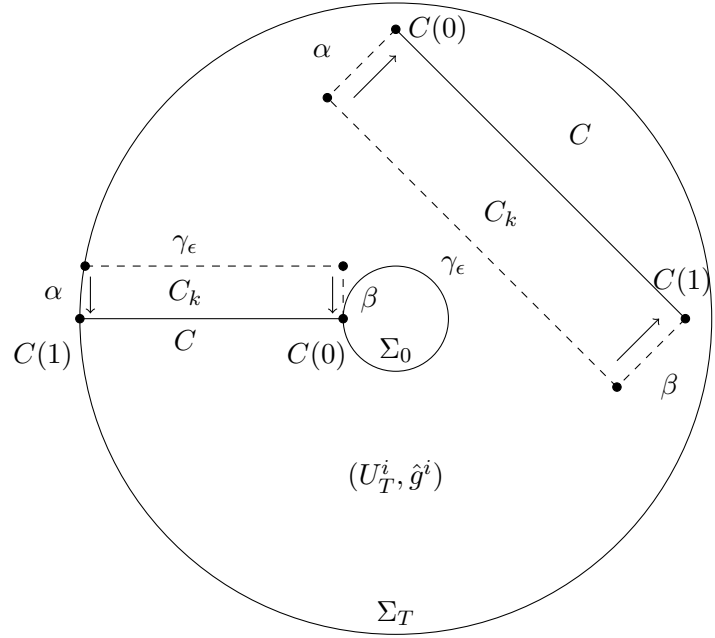

\begin{proof}

By Theorem \ref{gtog1} and Theorem \ref{g2tog3} we find
\begin{align}
\int_0^T\int_{\Sigma} \left|\frac{1}{H_i(x,t)^2}-\frac{r_0^2}{4}\right|^2 d\sigma dt &\le \int_0^T\int_{\Sigma} \left|\frac{1}{H_i(x,t)^2}-\frac{1}{\bar{H}_i(t)^2}\right|^2  d\sigma dt 
\\&+ \int_0^T\int_{\Sigma} \left|\frac{1}{\bar{H}_i(t)^2}-\frac{r_0^2}{4}\right|^2  d\sigma dt \rightarrow 0.
\end{align}

Now if we let $W \subset \Sigma \times [0,T]$ be a region foliated by straight lines parallel to $C(s)$, i.e. if $B_{\epsilon}(0) = \{\tau \in \R^2: |\tau| \le \epsilon\}$ then 
\begin{align}
W = \{\gamma_{\tau}:\tau \in B_{\epsilon}(0)\text{ and } \gamma_{\tau} \subset U_T^i \}
\end{align}
where $\gamma_{\tau}(s)=(T_{\tau}(s), \vec{\theta}_{\tau}(s))$ is a line parallel to $C(s)$ of distance $\tau$ away from $C(s)$ with respect to $\delta$. Then we can calculate
\begin{align}
&\int_{W} \left|\frac{1}{H_i(x,t)^2}-\frac{r_0^2}{4}\right|^2 dV 
\\&=\int_{-\epsilon}^{\epsilon} \int_{\gamma_{\tau}(s)} \left|\frac{1}{H_i(\vec{\theta}_{\tau}(s),T_{\tau}(s))^2}-\frac{r_0^2}{4}\right|^2ds d\tau \rightarrow 0.
\end{align}
Hence $\int_{\gamma_{\tau}(s)} \left|\frac{1}{H_i(\vec{\theta}_{\tau}(s),T_{\tau}(s))^2}-\frac{r_0^2}{4}\right|^2ds \rightarrow 0$ for a.e. $\tau \in B_{\epsilon}(0)$ on a subsequence. By adding segments of $\delta$ length $\le C \epsilon$ we can adjust $\gamma_{\tau}(s)$ to be a curve, $C_k(s)$, joining $C(0)$ to $C(1)$, where $\epsilon \rightarrow 0$ as $k \rightarrow \infty$. This can be done by adding $\alpha, \beta$ straight lines but if $p \in \partial U_T^i$ or $q \in \partial U_T^i$ then $\alpha$ or $\beta$ may need to be a length minimizing curve in $\partial U_T^i$ (See figure \ref{fig ApproximatingCurves}). All this together implies $L_{\delta}(C_k) \rightarrow L_{\delta}(C)$. Lastly we notice that
\begin{align}
|L_{\hat{g}^j}(C_k)-L_{\bar{g}^j}(C_k)| &= |L_{\hat{g}^j}(\alpha)-L_{\bar{g}^j}(\alpha)|
\\&+|L_{\hat{g}^j}(\gamma_{\epsilon})-L_{\bar{g}^j}(\gamma_{\epsilon})|+|L_{\hat{g}^j}(\beta)-L_{\bar{g}}(\beta)| \le 2 \bar{C} \epsilon
\end{align}
as $j \rightarrow \infty$ where the first and third term are smaller than $2 \bar{C} \epsilon$ since the difference in lengths between $\hat{g}^j$ and $\bar{g}^j$ is uniformly controlled by \eqref{MeanCurvBounds}, and $\alpha$ and $\beta$ are constructed to have length going to zero. The middle term goes to zero by Lemma \ref{IMCF-L}.
\end{proof}

\begin{lem} \label{IMCF-L-Sigma}
Fix a straight line with respect to $\delta$, $C(s)=(T(s), \vec{\theta}(s))$, parameterized on $[0,1]$ which is monotone in $t$
then
\begin{align}
|L_{\bar{g}^j}(C)-L_{\delta}(C)| \le \sqrt{T}C Diam(r_0^2 e^T \sigma)^2 \left(\int_{T(0)}^{T(1)} \left|g^i-r_0^2 e^t \sigma\right|_{\sigma}^2\,dt \right)^{1/4}.\label{lengthDistanceSigma}
\end{align}
\end{lem}
\begin{proof}
We start by computing the difference in the two lengths of $C$
\begin{eqnarray}
&&|L_{\bar{g}^j}(C)-L_{\delta}(C)|
\\&=& \int_0^1| \sqrt{ \frac{T'(t)^2r_0^2}{4} + g^i(\vec{\theta}',\vec{\theta}') }
- \sqrt{ \frac{T'(t)^2r_0^2}{4} + r_0^2e^t\sigma(\vec{\theta}',\vec{\theta}') }|\, ds \\
&\le& \int_0^1 \sqrt{  \left|g^i(\vec{\theta}',\vec{\theta}')-r_0^2 e^t \sigma(\vec{\theta}',\vec{\theta}') \right|}\, ds \\
&\le& \int_0^1 \sqrt{  \left|g^i-r_0^2 e^t \sigma \right|_{\sigma}} |\vec{\theta}'|_{\sigma}\, ds \\
&\le&\left( \int_0^1 \left|g^i-r_0^2 e^t \sigma\right|_{\sigma}\,ds \right)^{1/2}
\left(\int_0^1 |\vec{\theta}'|_{\sigma}^2\, ds \right)\label{Holder1Sigma}\\ 
&=&\left( \int_0^1  \left|g^i-r_0^2 e^t \sigma\right|_{\sigma}^2\,ds \right)^{1/4}
\left(\int_0^1 |\vec{\theta}'|_{\sigma}^2\, ds \right)\label{Holder2Sigma}
\end{eqnarray}
where we use Holder's inequality in \eqref{Holder1Sigma} and \eqref{Holder2Sigma}.
If $C(s)=(T(s), \vec{\theta}(s))$ and $T'(s) > 0$ everywhere, then we can reparametrize 
so that $T(s)=t$ and hence
\begin{eqnarray}
&&|L_{\bar{g}^j}(C)-L_{\delta}(C)| \\
&\le& \left(\int_{T(0)}^{T(1)} \left|g^i-r_0^2 e^t \sigma\right|_{\sigma}^2\,dt \right)^{1/4}
\left(\int_{T(0)}^{T(1)} |\vec{\theta}'|_{\sigma}^2\, dt\right)^{1/2}.
\end{eqnarray}
Lastly we notice that
\begin{align}
\left(\int_{T(0)}^{T(1)} |\vec{\theta}'|_{\sigma}^2\, dt\right)^{1/2} & \le \left(\int_0^T|\vec{\theta}'|_{\sigma}^2\, dt\right)^{1/2} \le \sqrt{T}C Diam(r_0^2 e^T \sigma)^2 
\end{align}
for some constant $C$.
\end{proof}

Again we are concerned that $\int_{T(0)}^{T(1)} \left|g^i-r_0^2 e^t \sigma\right|_{\sigma}^2dt$ may not go to zero so in a similar fashion as to Lemma \ref{LineApprox} we build a sequence of approximating curves.

\begin{lem}\label{LineApproxSigma}
Let $C(s)=(T(s), \vec{\theta}(s))$ parameterized on $[0,1]$ be a straight line in the annulus $(\Sigma\times [0,T], \delta)$ so that $\left(\int_{T(0)}^{T(1)} \left|g_i - r_0^2e^t \sigma\right|_{\sigma}^2\,ds \right)^{1/4}$ does not converge to $0$. Then there is a sequence of curves 
\begin{align}
C_k(s)=(T_k(s), \vec{\theta}_k(s))=
\begin{cases}
\alpha(s) & s \in [-\epsilon,0)
\\ \gamma_{\epsilon}(s) & s \in [0,1]
\\\beta(s)& s \in (1,1+ \epsilon]
\end{cases}
\end{align}
 where $\gamma_{\epsilon}$ is a line parallel to $C$, $\alpha, \beta$ are curves connecting the endpoints of $\gamma_{\epsilon}$ and $C_k$, and $L_{\delta}(C_k) \rightarrow L_{\delta}(C)$ as $k \rightarrow \infty$. If we assume
 \begin{align}
  g^i(x,t) \le C r_0^2 e^t \sigma \label{MetricUpperBound}
 \end{align}
 then
\begin{align}
\int_{T_k(0)}^{T_k(1)} \left|g^i - r_0^2e^t \sigma\right|_{\sigma}^2\,dt \rightarrow 0
\end{align}
and hence
\begin{align}
|L_{\bar{g}^i}(C_k)-L_{\delta}(C_k)| \le 2 \bar{C} \epsilon
\end{align}
as $i \rightarrow \infty$.
\end{lem}
\begin{proof}

By Theorem \ref{SPMTPrevious} we find
\begin{align}
\int_0^T\int_{\Sigma} \left|g^i - r_0^2e^t \sigma\right|_{\sigma}^2 d\sigma dt  \rightarrow 0.
\end{align}

Now if we let $W \subset \Sigma \times [0,T]$ be a region foliated by straight lines parallel to $C(s)$, i.e. if $B_{\epsilon}(0) = \{\tau \in \R^2: |\tau| \le \epsilon\}$ then 
\begin{align}
W = \{\gamma_{\tau}:\tau \in B_{\epsilon}(0)\text{ and } \gamma_{\tau} \subset U_T^i \}
\end{align}
where $\gamma_{\tau}(s)=(T_{\tau}(s), \vec{\theta}_{\tau}(s))$ is a line parallel to $C(s)$ of distance $\tau$ away from $C(s)$ with respect to $\delta$. Then we can calculate
\begin{align}
&\int_{W} \left|g^i - r_0^2e^t \sigma\right|_{\sigma}^2 dV =\int_{-\epsilon}^{\epsilon} \int_{\gamma_{\tau}(s)} \left|g_i - r_0^2e^t \sigma\right|_{\sigma}^2ds d\tau \rightarrow 0.
\end{align}
Hence 
\begin{align}
\int_{\gamma_{\tau}(s)} \left|g^i - r_0^2e^t \sigma\right|_{\sigma}^2ds \rightarrow 0
\end{align}
for a.e. $\tau \in B_{\epsilon}(0)$ on a subsequence. By adding segments of $\delta$ length $\le C \epsilon$ we can adjust $\gamma_{\tau}(s)$ to be a curve, $C_k(s)$, joining $C(0)$ to $C(1)$, where $\epsilon \rightarrow 0$ as $k \rightarrow \infty$. This can be done by adding $\alpha, \beta$ straight lines but if $p \in \partial U_T^i$ or $q \in \partial U_T^i$ then $\alpha$ or $\beta$ may need to be a length minimizing curve in $\partial U_T^i$ (See figure \ref{fig ApproximatingCurves}). All this together implies $L_{\delta}(C_k) \rightarrow L_{\delta}(C)$. Lastly we notice that
\begin{align}
|L_{\bar{g}^j}(C_k)-L_{\delta}(C_k)| &= |L_{\bar{g}^j}(\alpha)-L_{\delta}(\alpha)|
\\&+|L_{\bar{g}^j}(\gamma_{\epsilon})-L_{\delta}(\gamma_{\epsilon})|+|L_{\bar{g}^j}(\beta)-L_{\delta}(\beta)| \le 2 \bar{C} \epsilon
\end{align}
as $j \rightarrow \infty$ where the first and third term are smaller that $2 \bar{C} \epsilon$ since  $\bar{g}^j$ is uniformly controlled by \eqref{MetricUpperBound}, and $\alpha$ and $\beta$ are constructed to have length going to zero. The middle term goes to zero by Lemma \ref{IMCF-L}.
\end{proof}
\begin{rmrk}
Notice that we can combine Lemma \ref{LineApprox} and Lemma \ref{LineApproxSigma} by choosing a $C_k$ which works in both situations. This will be important for the next Corollary.
\end{rmrk}

\begin{Cor}\label{limsupEst}
Assume that 
 \begin{align} 
  0 < H_0 &\le H_i(x,t) 
\\ g^i(x,t) &\le C r_0^2 e^t \sigma 
\\ \hat{g}^j &\rightarrow \bar{g}^j \text{ in } L^2,
\\ \bar{g}^j &\rightarrow \delta \text{ in } L^2,
\\ g^i(\cdot,0) &\rightarrow r_0^2 \sigma \text{ in } C^0
 \end{align}
then for any $p,q \in \Sigma\times [0,T]$ we find
\begin{align}
\limsup_{j \rightarrow \infty} d_{\hat{g}^j}(p,q) \le d_{\delta}(p,q).
\end{align}
\end{Cor}
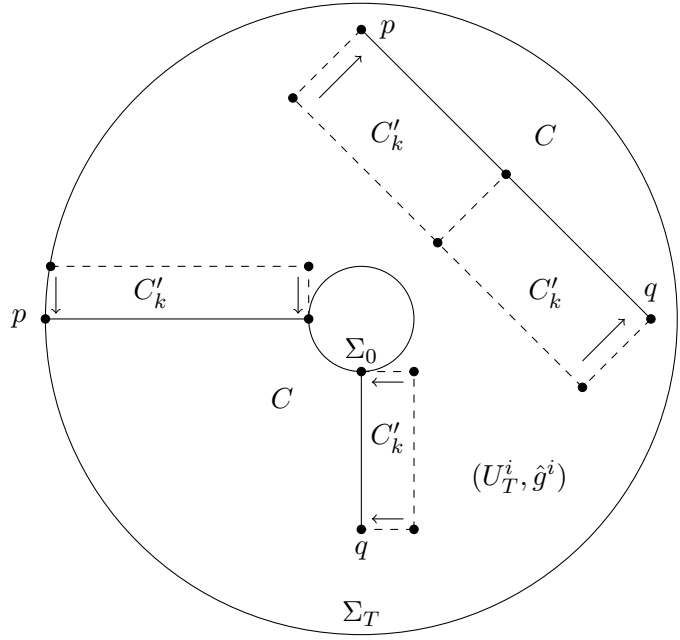
\begin{figure}[H]\label{fig ApproximatingCurves2}
\begin{tikzpicture}[scale=.7]
\draw (0,0) circle (1cm);
\draw (0,0) circle (6cm);
\draw (3,-3) node{$(U_T^i,\hat{g}^i)$};
\draw (0,-.6) node{$\Sigma_0$};
\draw (0,-5.6) node{$\Sigma_T$};
\draw (.5,5.5) node{$p$};
\draw[fill=black] (0,5.5) circle (.08cm);
\draw (5.5,.5) node{$q$};
\draw[fill=black] (5.5,0) circle (.08cm);
\draw (0,5.5) -- (5.5,0);
\draw[fill=black] (-1.3,4.2) circle (.08cm);
\draw[fill=black] (4.2,-1.3) circle (.08cm);
\draw[fill=black] (1.45,1.45) circle (.08cm);
\draw[fill=black] (2.75,2.75) circle (.08cm);
\draw[dashed] (1.45,1.45) -- (2.75,2.75);
\draw[dashed] (5.5,0) -- (4.2,-1.3);
\draw[dashed] (0,5.5) -- (-1.3,4.2);
\draw[dashed] (4.2,-1.3) -- (-1.3,4.2);
\draw[->] (-.8,4.2) -- (0,5);
\draw[->] (4.2,-.8) -- (5,0);
\draw (.5,3.5) node{$C_k'$};
\draw (3.5,.5) node{$C_k'$};
\draw (3.5,3.5) node{$C$};
\draw (-6.5,0) node{$p$};
\draw (-1,0) -- (-6,0);
\draw[fill=black] (-1,0) circle (.08 cm);
\draw[fill=black] (-6,0) circle (.08 cm);
\draw[fill=black] (-1,1) circle (.08 cm);
\draw[fill=black] (0,-1) circle (.08 cm);
\draw[fill=black] (0,-4) circle (.08 cm);
\draw[fill=black] (1,-1) circle (.08 cm);
\draw[fill=black] (1,-4) circle (.08 cm);
\draw[fill=black] (-5.9,1) circle (.08 cm);
\draw[dashed] (-5.9,1) -- (-1,1);
\draw[dashed] (-1,0) --(-1,1);
\draw[dashed] (1,-4) -- (1,-1);
\draw[dashed] (0,-1) -- (1,-1);
\draw[dashed] (1,-4) -- (0,-4);
\draw (0,-1) -- (0,-4);
\draw (-4,.5) node{$C_k'$};
\draw (.5, -2.25) node{$C_k'$};
\draw (-1.5,-1.5) node{$C$};
\draw (0,-4.4) node{$q$};
\draw[->] (-5.8,.8) -- (-5.8,.1);
\draw[->] (-1.2,.8) -- (-1.2,.1);
\draw[->] (.8,-3.8) -- (.2,-3.8);
\draw[->] (.8,-1.2) -- (.2,-1.2);
\end{tikzpicture}
\caption{Two families of curves $C_k'$ approximating the curve $C$ which realizes the distance between $p,q \in U_T^i$ with respect to $\delta$. }
\end{figure}
\begin{proof}
Let $C(t)=(T(s), \vec{\theta}(s))$ parameterized on $[0,1]$ be the minimizing geodesic with respect to $(\Sigma\times [0,T], \delta)$ between the points $p,q \in \Sigma \times (0,T]$. Then we know that $C(t)$ is a straight line in the annulus $(\Sigma\times [0,T], \delta)$, or can be broken down into two pieces which are straight lines and one piece which lies completely in $\Sigma_0$. 

For each straight line $C_l$, possibly further decomposing into straight lines which are monotone in $t$, we can apply Lemma \ref{LineApprox} and Lemma \ref{LineApproxSigma} to find a sequence of curves $C_k(t)$ so that $L_{\delta}(C_k) \rightarrow L_{\delta}(C_l)$. For the portion of the curve tangent to $\Sigma_0$ we use the fact that $\hat{g}^i$ and $\bar{g}^i$ agree on this set and $g_i(\cdot,0) \rightarrow r_0^2 \sigma$ in $C^0$. Putting this altogether we obtain an approximating curve $C'_k$ to find
\begin{align}
&d_{\hat{g}^j}(p,r)\le L_{\hat{g}^j}(C_k')
\\&= L_{\hat{g}^j}(C_k') -  L_{\bar{g}^j}(C_k')+ L_{\bar{g}^j}(C_k')- L_{\delta}(C_k')+ L_{\delta}(C_k')
\\&\le |L_{\hat{g}^j}(C_k') -  L_{\bar{g}^j}(C_k')| + |L_{\bar{g}^j}(C_k')- L_{\delta}(C_k')| + L_{\delta}(C_k')
\\&= |L_{\hat{g}^j}(C_k') -  L_{\bar{g}^j}(C_k')| + |L_{\bar{g}^j}(C_k')- L_{\delta}(C_k')| + L_{\delta}(C) + 4 C \epsilon
\\&= |L_{\hat{g}^j}(C_k') -  L_{\bar{g}^j}(C_k')| + |L_{\bar{g}^j}(C_k')- L_{\delta}(C_k')| +  4 C \epsilon + d_{\delta}(p,r).
\end{align}
So by taking limits and using Lemma \ref{LineApprox} and Lemma \ref{IMCF-L} we find that
\begin{align}
\limsup_{j \rightarrow \infty} d_{\hat{g}^j}(p,r) \le  d_{\delta}(p,r)+  2C \epsilon
\end{align}
and since this is true $\forall \epsilon >0$ the result follows.
\end{proof}

\subsection{Consequences of $C^0$ Assumptions}

It was noticed in \cite{BS} that a $C^0$ lower bound on $\hat{g}^i$ in terms of the limiting metric, in this case $\delta$, is needed in order to imply uniform, GH and SWIF convergence from $L^2$ convergence. In this subsection we deduce the consequence of this $C^0$ lower bound which is analogous to Lemma 4.1 in \cite{BS}.

\begin{lem}\label{distLowerBound} Let $p,q \in [0,T]\times \Sigma$ and assume that 
\begin{align}
\frac{1}{H_j(x,t)^2} &\ge \frac{r_0^2}{4}e^t- \frac{1}{j},
\\ diam(U_T,\hat{g}^j) &\le D
\end{align}
 then
 \begin{align}
 d_{\hat{g}_j}(p,q) - d_{\bar{g}^j}(p,q) &\ge -\frac{D}{\sqrt{j}}
 \end{align}
 and hence
\begin{align}
\liminf_{j \rightarrow \infty} \left(d_{\hat{g}_j}(p,q) - d_{\bar{g}_j}(p,q)\right)\ge 0.
\end{align}
\end{lem}
\begin{proof}
Let $C_j(s)=(T_j(s),\vec{\theta}_j(s))$ be the minimizing geodesic in $U_T^j$, parameterized by arc length w.r.t. $\hat{g}_j$, realizing the distance between $p$ and $q$ then compute
\begin{align}
 d_{\hat{g}_j}(p,q) &= \int_0^{L_{\hat{g}_j}(C_j)} \sqrt{\frac{T_j'(s)^2}{H_j(\vec{\theta}_j'(s),T_j(s))^2}+g^j(\vec{\theta_j}'(s),\vec{\theta_j}'(s))}ds
 \\&\ge \int_0^{L_{\hat{g}_j}(C_j)}  \sqrt{T_j'(s)^2\left ( \frac{r_0^2}{4}e^s- \frac{1}{j} \right )+g^j(\vec{\theta_j}'(s),\vec{\theta_j}'(s))}ds
  \\&\ge \int_0^{L_{\hat{g}_j}(C_j)} \sqrt{T_j'(s)^2\frac{r_0^2}{4}e^s+g^j(\vec{\theta_j}'(t),\vec{\theta_j}'(t))} - \frac{|T_j'(s)|}{\sqrt{j}}ds
  \\&\ge L_{\bar{g}^j}(C_j)-\frac{1}{\sqrt{j}}\int_0^{L_{\hat{g}_j}(C_j)}|T_j'(s)|ds \ge d_{\bar{g}^j}(p,q)-\frac{D}{\sqrt{j}}
\end{align}
where we used the inequality $\sqrt{|a-b|} \ge |\sqrt{a} - \sqrt{b}| \ge \sqrt{a} -\sqrt{b}$ in the third line. The last result follows by taking limits.
\end{proof}

\begin{lem}\label{distLowerBoundSigma} Let $p,q \in [0,T]\times \Sigma$ and assume that 
\begin{align}
\left( 1 - \frac{C}{j}\right)r_0^2e^t \sigma &\le g^i(x,t),\label{MetricAssumption}
\\ diam(U_T,\hat{g}^j) &\le D
\end{align}
 then
  \begin{align}
 d_{\bar{g}_j}(p,q) - d_{\delta}(p,q) &\ge -\frac{D}{\sqrt{j}\left(1-\frac{C}{j}\right)}.
 \end{align}
 and hence
\begin{align}
\liminf_{j \rightarrow \infty} \left(d_{\bar{g}_j}(p,q) - d_{\delta}(p,q)\right)\ge 0.
\end{align}
\end{lem}
\begin{proof}
Let $C_j(s)=(T_j(s),\vec{\theta}_j(s))$ be the minimizing geodesic with respect to $\bar{g}^j$, parameterized by arc length w.r.t. $\bar{g}_j$, realizing the distance between $p$ and $q$ then compute
\begin{align}
 &d_{\bar{g}_j}(p,q) = \int_0^{L_{\bar{g}_j}(C_j)} \sqrt{\frac{r_0^2e^sT_j'(s)^2}{4}+g^j(\vec{\theta_j}'(s),\vec{\theta_j}'(s))}ds
 \\&\ge \int_0^{L_{\bar{g}_j}(C_j)}  \sqrt{\frac{r_0^2e^sT_j'(s)^2}{4}+\left( 1 - \frac{C}{j}\right)r_0^2 e^s \sigma(\vec{\theta_j}'(s),\vec{\theta_j}'(s))}ds\label{MetricAssumUsed1}
  \\&\ge \int_0^{L_{\bar{g}_j}(C_j)} \sqrt{\frac{r_0^2e^sT_j'(s)^2}{4}+r_0^2e^s\sigma(\vec{\theta_j}'(s),\vec{\theta_j}'(s))}ds\label{sqrtIneg1}
  \\& -  \int_0^{L_{\bar{g}_j}(C_j)}\frac{r_0e^{s/2}|\sigma(\vec{\theta_j}'(s),\vec{\theta_j}'(s))|}{\sqrt{j}}ds\label{sqrtIneq2}
  \\&\ge L_{\delta}(C_j)-\frac{1}{\sqrt{j}\left(1-\frac{C}{j}\right)}\int_0^{L_{\bar{g}_j}(C_j)}|g^j(\vec{\theta_j}'(s),\vec{\theta_j}'(s))|ds \label{MetricAssumUsed2}
  \\&\ge d_{\delta}(p,q)-\frac{D}{\sqrt{j}\left(1-\frac{C}{j}\right)}
\end{align}
where we used the inequality $\sqrt{|a-b|} \ge |\sqrt{a} - \sqrt{b}| \ge \sqrt{a} -\sqrt{b}$ in \eqref{sqrtIneg1} and \eqref{sqrtIneq2}. We used \eqref{MetricAssumption} in lines \eqref{MetricAssumUsed1} and \eqref{MetricAssumUsed2}. The last result follows by taking limits.
\end{proof}

\begin{Cor}\label{distLowerBound2} Let $p,q \in [0,T]\times \Sigma$ and assume that 
\begin{align}
\frac{1}{H_j(x,t)^2} &\ge \frac{r_0^2}{4}e^t- \frac{1}{j}, 
\\ diam(U_T,\bar{g}^j) &\le D, 
\\ \left( 1 - \frac{C}{j}\right) r_0^2 e^t\sigma &\le g^i(x,t)\le Cr_0^2 e^t \sigma,\label{C0metricbounds}
\end{align}
then
\begin{align}
\liminf_{j \rightarrow \infty} d_{\hat{g}_j}(p,q) \ge d_{\delta}(p,q).
\end{align}
\end{Cor}
\begin{proof}
By Lemma \ref{distLowerBound} we know
\begin{align}
 d_{\hat{g}_j}(p,q) - d_{\bar{g}^j}(p,q) &\ge -\frac{D}{\sqrt{j}}
 \end{align}
 and by Lemma \ref{distLowerBoundSigma} we know
 \begin{align}
 d_{\bar{g}_j}(p,q) - d_{\delta}(p,q) &\ge-\frac{D}{\sqrt{j}\left(1-\frac{C}{j}\right)}.
 \end{align}
 The result follows by combining these two inequalities and taking limits.
\end{proof}
\begin{rmrk}
Notice that if 
\begin{align}
\|g^i - r_0^2e^t \sigma\|_{C^0} &\le \frac{C}{j} \text{ } \forall (x,t) \in \Sigma \times [0,T]
\end{align}
then \eqref{C0metricbounds} is satisfied and hence the corollary applies to this case as well.
\end{rmrk}

\subsection{Proof of Convergence}\label{subsec:ConvProof}

Now we are able to show uniform, GH, and Flat convergence for the desired metrics.
 \begin{thm}\label{IMCFConv}
 Assume 
 \begin{align}
 &0 < H_0 \le H^j(x,t)\le H_1 < \infty,
 \\ &0 < \left( 1 - \frac{C_0}{j}\right) r_0^2 e^t \sigma \le g^i(x,t) \le C_1 r_0^2 e^t \sigma, 
 \\&\frac{1}{H_j(x,t)^2} \ge \frac{r_0^2}{4}e^t- \frac{1}{j}, 
 \\&\text{diam}(U_T,\hat{g}^j) \le D, 
 \\& \hat{g}^j \rightarrow \bar{g}^j \text{ in } L^2,\text{ and }
 \\&\bar{g}^j \rightarrow \delta \text{ in } L^2,
 \end{align}
  then $\hat{g}$ converges uniformly to $\delta$ as well as
  \begin{align}
  (U_T^j,\hat{g}^j) &\GHto (\Sigma\times[0,T],\delta),
  \\(U_T^j,\hat{g}^j) &\Fto (\Sigma\times[0,T],\delta).
  \end{align} 
 \end{thm}
 \begin{proof}
 Define $\delta_c = dt^2 + \sigma$ and by the assumptions on $H^j(x,t)$ and $g^i(x,t)$ we find that 
 \begin{align}
\min(\frac{1}{H_1^2},c_0) \delta_c \le \frac{1}{H_1^2}dt^2 + c_0 \sigma \le  \hat{g}^i \le \frac{1}{H_0^2}dt^2 + c_1 \sigma \le \max(\frac{1}{H_0^2},c_1) \delta_c
 \end{align}
 where $c_0 = \left( 1 - C_0\right) r_0^2$ and $c_1 = C_1 r_0^2 e^T$.
 Hence for 
 \begin{align}
 \lambda = \max(H_0^{-2},c_1, H_1^2, c_0^{-1})
 \end{align}
 we can find the Lipschitz bounds
 \begin{align}
 \frac{1}{\lambda} \le \frac{\hat{g}^i(p,q)}{\delta_c(p,q)} \le \lambda.
 \end{align}
 Now we can apply Theorem \ref{HLS-thm} to conclude that a subsequence $\hat{g}^j$ converges in the uniform, GH and Flat sense to some length metric $g_{\infty}$ so that
  \begin{align}
 \frac{1}{\lambda} \le \frac{g_{\infty}(p,q)}{\delta_c(p,q)} \le \lambda.
 \end{align}
 Now our goal is to show that $g_{\infty} = \delta$ by observing that $\displaystyle \lim_{j \rightarrow \infty} d_{\hat{g}^j}(p,q) = \delta(p,q)$ pointwise. To this end, let $p,q \in [0,T]\times \Sigma$ so by Corollary \ref{distLowerBound2} we have
\begin{align}\label{distliminf}
\liminf_{j \rightarrow \infty} d_{\hat{g}_j}(p,q) \ge d_{\delta}(p,q),
\end{align}
and by Corollary \ref{limsupEst} we have
\begin{align}
\limsup_{j \rightarrow \infty} d_j(p,q)\le  d_{\delta}(p,q),
\end{align}
and hence we find
\begin{align}
\lim_{j \rightarrow \infty} d_{\hat{g}_j}(p,q) = d_{\delta}(p,q),
\end{align}
which gives pointwise convergence of distances and hence $g_{\infty} = \delta$. We can get rid of the need for subsequences by noticing that every subsequence of the original sequence converges to the same limit.
 \end{proof}
 
 \section{Stability of PMT and RPI} \label{sect-Stability}
 
 In the last section we showed what hypotheses we need in order to show uniform, GH and SWIF convergence of $\hat{g}^i$ to $\delta$. In this section we want to show how to use the results of section \ref{sect-Hawk} to obtain the hypotheses of section \ref{sect-IMCF} in order to prove the main theorems of this paper.
 
\subsection{Obtaining the Hypotheses of Theorem \ref{IMCFConv}} \label{subsec: Obtaining Hypotheses}

We first want to show that $g^i \rightarrow r_0^2 e^t \sigma$ uniformly which is where the curvature assumptions of the main theorems come into play. In order to take advantage of these curvature assumptions in combination with Lemma \ref{GoToZero} and Lemma \ref{WeakRicciEstimate} we will use the following result.

\begin{thm}[Theroem 78 \cite{P}]\label{Rigidity2}
Given $n \ge 2,$ $v,D \in (0,\infty),$ and $\lambda \in \R,$ there is an $\epsilon = \epsilon(n, \lambda, v, D) >0$ such that any closed Riemannian $n-$manifold $(M,g)$ with
\begin{align}
diam(M) &\le D,
\\Vol(M) &\ge v,
\\|K-\lambda| &\le \epsilon,
\end{align}
where $K$ is the sectional curvature of $(M,g)$, is $C^{1,\alpha}$ close to a metric of constant curvature $\lambda$.
\end{thm}

We will also need an integral pinching theorem that was used in the author's previous paper \cite{BA2}.

\begin{Cor}(Petersen and Wei \cite{PW})\label{Rigidity}
Given any integer $n \ge 2$, and numbers $p > n/2$, $\lambda \in \R$, $v >0$, $D < \infty$, one can find $\epsilon = \epsilon(n,p,\lambda, D) > 0$ such that a closed Riemannian $n-$manifold $(\Sigma,g)$ with
\begin{align}
&\text{vol}(\Sigma)\ge v
\\&\text{diam}(\Sigma) \le D
\\& \frac{1}{|\Sigma|} \int_{\Sigma} \|R - \lambda g \circ g\|^p d \mu \le \epsilon(n,p,\lambda,D)\label{Eq-l2curv}
\end{align}
is $C^{\alpha}$, $\alpha < 2 - \frac{n}{p}$ close to a constant curvature metric on $\Sigma$.
\end{Cor}
In our case $n=2$, $p = 2$, $\alpha < 1$ and the Riemann curvature tensor is $R = K g \circ g$, where $g \circ g$ represents the Kulkarni-Nomizu product, and so $\|R - \lambda g \circ g\|^2 = \|g \circ g\|^2 |K - \lambda|^2 = 2^4|K - \lambda|^2$.

We will apply this in the case where $n=2$ and $\Sigma$ has the topology of the sphere so we know that the constant curvature metric which $\Sigma_t^i$ will converge to is the round sphere. We now apply this result in order to obtain $C^{\alpha}$ convergence of $\Sigma_t^i$ to $S^2(r_0e^{t/2})$.
 
 \begin{thm}\label{ConvForEachSlicet}
 Let $U_{T}^i \subset M_i^3$ be a sequence s.t. $U_{T}^i\subset \mathcal{M}_{r_0,H_0,I_0}^{T,H_1,A_1}$ and $m_H(\Sigma_{T}^i) \rightarrow 0$ as $i \rightarrow \infty$. If we assume that 
 \begin{align}
 |K^i|_{C^{0,\alpha}(\Sigma\times[0,T])} &\le C
\\|Rc^i(\nu,\nu)|_{C^{0,\alpha}(\Sigma\times[0,T])} &\le C
 \\ |R^i|_{C^{0,\alpha}(\Sigma\times[0,T])} &\le C
\\diam(\Sigma_t^i)&\le D\text{ }\forall t \in [0,T]
 \end{align}
 then
 \begin{align}
g^i \rightarrow r_0^2 e^t \sigma 
 \end{align}
  in $C^{1,\alpha}(\Sigma\times \{t\})$ uniformly for every $t \in [0,T]$, as well as, 
  \begin{align}\label{RicciLowerBound}
 Rc^i(\nu,\nu)&\ge - \frac{C}{i}.
 \end{align} 
 This implies that
  \begin{align}
  \bar{g}^i \rightarrow \delta
  \end{align}
  in $C^{1,\alpha}(\Sigma\times  [0,T])$.
 \end{thm}
 \begin{proof}
By Lemma \ref{WeakRicciEstimate} we know that $Rc^j(\nu,\nu) \rightharpoonup 0$ on $\Sigma\times[0,T]$. By combining with the assumption $|Rc^i(\nu,\nu)|_{C^{0,\alpha}(\Sigma\times[0,T])} \le C$ we can conclude by the Arzela-Ascoli Theorem that on a subseqeunce $Rc^j(\nu,\nu) \rightarrow 0$ uniformly over $\Sigma \times [0,T]$. Note that this implies the lower bound in equation \eqref{RicciLowerBound}. By Corollary \ref{GoToZero} combined with the assumption $|R^i|_{C^{0,\alpha}([0,T])} \le C$ we know that on a subsequence $R^j \rightarrow 0$ uniformly over $\Sigma\times[0,T]$.  Then since $K_{12}^j= R^j-2Rc^j(\nu,\nu)$ we find that $K_{12}\rightarrow 0$ uniformly over $\Sigma\times[0,T]$. 
  
Lastly, we note that $K^j = \lambda_1^j\lambda_2^j + K_{12}^j$  which implies 
\begin{align}
|K^j - \lambda_1^j\lambda_2^j| = |K_{12}^j| \rightarrow 0
\end{align}
uniformly over $\Sigma \times [0,T]$. Now by combining with Lemma \ref{GoToZero} we find 
\begin{align}
|K^j - \frac{e^{-t}}{r_0^2}| \le |K^j - \lambda_1^j\lambda_2^j| + |\lambda_1^j\lambda_2^j- \frac{e^{-t}}{r_0^2}| \rightarrow 0 \label{GaussConvergence}
\end{align}
pointwise for a.e. $(x,t) \in \Sigma \times [0,T]$ since $\lambda_p^j \rightarrow \frac{e^{-t/2}}{r_0}$, $p=1,2$, pointwise a.e. on a subsequence. By combining with the assumption on $K^i$ we find
 \begin{align}
|K^j - \frac{e^{-t}}{r_0^2}| \rightarrow 0
\end{align}
uniformly over $\Sigma \times [0,T]$. By combining with the diameter bound diam$(\Sigma_t^i)\le D$ $\forall t \in [0,T]$  and the fact that we know $|\Sigma_t^i| = 4 \pi r_0^2 e^t$ we can apply the pinching result Theorem \ref{Rigidity2} which implies that $|g^j(x,t) - r_0^2e^t\sigma(x)|_{C^{1,\alpha}} \rightarrow 0$ uniformly for each $(x,t) \in \Sigma\times[0,T]$ as $j \rightarrow \infty$ where $\alpha < 1$. Note that this immediately implies that $\bar{g}^i \rightarrow \delta$ by the definition of $\bar{g}^i$.

Then we can get rid of the need for a subsequence by assuming to the contrary that for $\epsilon > 0$ there exists a subsequence so that $|\bar{g}^j - \delta|_{C^{1,\alpha}} \ge \epsilon$, but this subsequence satisfies the hypotheses of Theorem \ref{SPMT} and hence by what we have just shown we know a further subsequence must converge which is a contradiction.
\end{proof}

 \begin{thm}\label{ConvForZeroSlice}
 Let $U_{T}^i \subset M_i^3$ be a sequence s.t. $U_{T}^i\subset \mathcal{M}_{r_0,H_0,I_0}^{T,H_1,A_1}$ and $m_H(\Sigma_{T}^i) \rightarrow 0$ as $i \rightarrow \infty$. If we assume that 
 \begin{align}
|Rc^i(\nu,\nu)|_{C^{0,\alpha}(\Sigma\times[0,T])} &\le C
\\diam(\Sigma_0^i)&\le D\text{ }\forall t \in [0,T]
 \end{align}
 then
 \begin{align}
g^i \rightarrow r_0^2 e^t \sigma 
 \end{align}
  in $C^{0,\alpha}(\Sigma\times \{t\})$. 
 \end{thm}
 \begin{proof}
 The proof of this theorem is extremely similar to Theorem \ref{ConvForEachSlicet} where the steps are repeated exactly up to \eqref{GaussConvergence}. At this point we again use that $\lambda_p^j \rightarrow \frac{e^{-t/2}}{r_0}$, $p=1,2$, pointwise a.e. on a subsequence but if $t=0$ is a time where we do not have convergemce of $\lambda_p^j$ then we choose a time which is arbitrarily close to $t=0$ as the new starting time. Then we complete the argument by applying the integral pinching result Theorem \ref{Rigidity}.
 \end{proof}

Now we state a straight forward corollary for completeness.

\begin{Cor}\label{sigmaControl}
If $g^i \rightarrow r_0^2 e^t \sigma$ in $C^{0,\alpha}(\Sigma\times \{t\})$ uniformly for every $t \in [0,T]$ then
\begin{align}
0 < c_0 \sigma &\le g^i(x,t) \le c_1 \sigma.
 \end{align}
\end{Cor}
\begin{proof}
This follows immediately from the uniform $C^0$ convergence.
\end{proof}

Lastly we show how to obtain diameter control of $U_T^i$ from control on $g^i$ and the assumptions on mean curvature.

\begin{lem}\label{DiamControl}
If we assume that
\begin{align}
 g^i(x,t) \le c_1 \sigma,
 \end{align}
 then if we define $\delta_c = dt^2 + \sigma$ we find
 \begin{align}
diam(U_T^i, \hat{g}^i)&\le \max(H_0^{-1},c_1,1) diam(\Sigma \times [0,T],\delta_c)
 \end{align}
\end{lem}
\begin{proof}
Let $C(t) = (T(t),\vec{\theta(t)})$ be the curve which realizes the diameter of $(U_T^i,\hat{g}^i)$ between the points $p,q \in U_T^i$ and let $\gamma(t)$ be the curve which minimizes the distance between $p,q$ with respect to $\delta_c = dt^2 + \sigma$. Now notice
\begin{align}
diam(U_T^i, \hat{g}^i)&= d_{\hat{g}^i}(p,q) =L_{\hat{g}^i}(C) \le L_{\hat{g}^i}(\gamma)
\\&= \int_0^1 \sqrt{\frac{T'(t)^2}{H(x,t)^2} + g^i(\theta'(t),\theta'(t))}dt
\\&\le \int_0^1 \sqrt{\frac{T'(t)^2}{H_0^2} + c_1\sigma(\theta'(t),\theta'(t))}dt
\\&\le \max(H_0^{-1},c_1,1) \int_0^1 \sqrt{T'(t)^2 +\sigma(\theta'(t),\theta'(t))}dt
\\&= \max(H_0^{-1},c_1,1)L_{\delta_c}(\gamma)
\\&= \max(H_0^{-1},c_1,1)d_{\delta_c}(p,q)
\\&\le \max(H_0^{-1},c_1,1) diam(\Sigma \times [0,T],\delta_c)
\end{align}
\end{proof}

\subsection{Proof of Main Theorems}
In this subsection we will finish the proofs of Theorem \ref{SPMT} and Theorem \ref{SPMTJumpRegion} which we note are fairly quick proofs at this point since we have organized the important results in subsection \ref{subsec: Obtaining Hypotheses}.
\vspace{.5cm}

\textbf{Proof of  Theorem \ref{SPMT}:}
 \begin{proof}
 By Theorem \ref{ConvForEachSlicet}, Corollary \ref{C0BoundBelow}, Corollary \ref{sigmaControl}, and Lemma \ref{DiamControl}, combined with the $L^2$ convergence results of \cite{BA2} we have the hypthotheses necessary to apply Theorem \ref{IMCFConv} which completes the proof of Theorem \ref{SPMT}.
 \end{proof}
 
 \textbf{Poof of Theorem \ref{SPMTLessCurv}:}
 \begin{proof}
  By Theorem \ref{ConvForZeroSlice}, Corollary \ref{C0BoundBelow}, assumption \ref{MetricBounds}, and Lemma \ref{DiamControl}, combined with the $L^2$ convergence results of \cite{BA2} we have the hypthotheses necessary to apply Theorem \ref{IMCFConv} which completes the proof of Theorem \ref{SPMT}.
 \end{proof}

 \textbf{Proof of  Theorem \ref{SPMTJumpRegion}:}
 \begin{proof}
 First we notice that the assumptions of Theorem \ref{SPMTJumpRegion} allow us to apply Theorem \ref{SPMT} as well as Lemma \ref{VolumeEstimate}, Lemma \ref{DiameterEstimate} and Lemma \ref{DistancePreservingEstimate}. All of these results give us the necessary hypotheses to apply Theorem \ref{SWIFonCompactSets} which completes the proof of Theorem \ref{SPMTJumpRegion}.
 \end{proof}
 
  \textbf{Poof of Theorem \ref{SPMTJumpRegionWeakCurv}:}
 \begin{proof}
 First we notice that the assumptions of Theorem \ref{SPMTJumpRegionWeakCurv} allow us to apply Theorem \ref{SPMTLessCurv} as well as Lemma \ref{VolumeEstimate}, Lemma \ref{DiameterEstimate} and Lemma \ref{DistancePreservingEstimate}. All of these results give us the necessary hypotheses to apply Theorem \ref{SWIFonCompactSets} which completes the proof of Theorem \ref{SPMTJumpRegionWeakCurv}.
 \end{proof}
 
 \section{Examples and Ideas for Further Study} \label{sec Example}
 
 In this section we give three examples which illustrate the hypotheses assume in the main theorems as well as discuss possible uses of the stability theorems of the author in combination with further understanding of the properties of weak and smooth IMCF.
 
\begin{ex}\label{NoGHExample}
In the paper by Lee and Sormani \cite{LS1} an example is given of a manifold whose mass is going to zero with increasingly many, increasingly thin gravity wells of a fixed depth that has a SWIF limit but does not converge in the GH sense (See Example 5.6 of \cite{LS1}). In this paper we adapt this example, depicted in figure \ref{fig IncreasinglyManyWellsExample}, in order to explore the hypotheses of Theorem \ref{SPMTJumpRegion} for an example which does not converge in the GH sense. 

\begin{figure}[H]\label{fig IncreasinglyManyWellsExample}
\begin{tikzpicture}[scale=.5]
\draw (-8,4.5) node{$U_{T_1}^1$};
\draw (-8,0) circle (1.8cm);
\draw[dotted] (-8,0) circle (1.3cm);
\draw[dotted] (-8,0) circle (2.3cm);
\draw (-8,0) circle (3.6cm);
\draw[fill=black] (-9.8,0) circle (.15cm);
\draw[fill=black] (-6.2,0) circle (.15cm);
\draw[dotted] (-9.8,0) circle (.35cm);
\draw[dotted] (-6.2,0) circle (.35cm);
\draw (0,4.5) node{$U_{T_2}^2$};
\draw (0,0) circle (1.8cm);
\draw (0,0) circle (3.6cm);
\draw[dotted] (0,0) circle (1.3cm);
\draw[dotted] (0,0) circle (2.3cm);
\draw[fill=black] (-1.8,0) circle (.1cm);
\draw[fill=black] (1.8,0) circle (.1cm);
\draw[dotted] (-1.8,0) circle (.3cm);
\draw[dotted] (1.8,0) circle (.3cm);
\draw[fill=black] (0,-1.8) circle (.1cm);
\draw[fill=black] (0,1.8) circle (.1cm);
\draw[dotted] (0,-1.8) circle (.3cm);
\draw[dotted] (0,1.8) circle (.3cm);
\draw (8,4.5) node{$U_{T_3}^3$};
\draw (8,0) circle (1.8cm);
\draw (8,0) circle (3.6cm);
\draw[dotted] (8,0) circle (1.3cm);
\draw[dotted] (8,0) circle (2.3cm);
\draw[fill=black] (8,1.8) circle (.05cm);
\draw[dotted] (8,1.8) circle (.25cm);
\draw[fill=black] (8,-1.8) circle (.05cm);
\draw[dotted] (8,-1.8) circle (.25cm);
\draw[fill=black] (6.2,0) circle (.05cm);
\draw[dotted] (6.2,0) circle (.25cm);
\draw[fill=black] (9.8,0) circle (.05cm);
\draw[dotted] (9.8,0) circle (.25cm);
\draw[fill=black] (8+1.8*.707,1.8*.707) circle (.05cm);
\draw[dotted] (8+1.8*.707,1.8*.707) circle (.25cm);
\draw[fill=black] (8-1.8*.707,1.8*.707) circle (.05cm);
\draw[dotted] (8-1.8*.707,1.8*.707) circle (.25cm);
\draw[fill=black] (8+1.8*.707,-1.8*.707) circle (.05cm);
\draw[dotted] (8+1.8*.707,-1.8*.707) circle (.25cm);
\draw[fill=black] (8-1.8*.707,-1.8*.707) circle (.05cm);
\draw[dotted] (8-1.8*.707,-1.8*.707) circle (.25cm);
\draw (13,0) node{$\cdot\cdot\cdot$};
\end{tikzpicture}
\caption{Sequence of manifolds with increasingly thin gravity wells of a fixed depth centered at the points indicated.}
\end{figure}
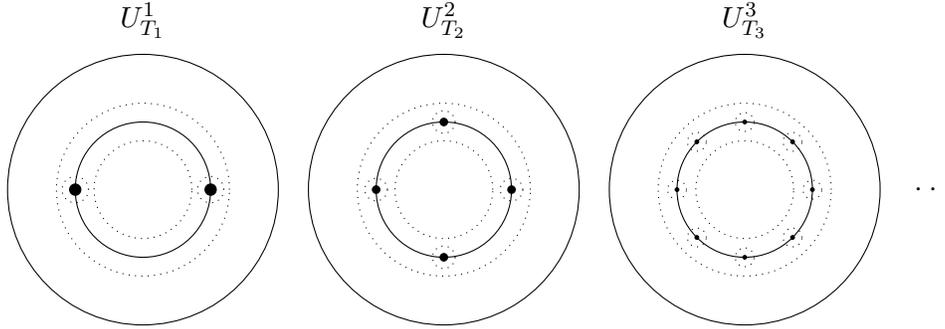

Here we consider the smooth solution of IMCF of $\Sigma_0^i$ inside the manifolds $M_i$, which we will define in a similar way as Lee and Sormani \cite{LS1} Example 5.6, so that $M_i$ is asymptotically flat manifolds with ADM mass $1/i$. In the process of constructing these examples Lee and Sormani show that a rotationally symmetric manifold exists with constant positive sectional curvature $K_i$ on annular regions, which are depicted by dotted circles above in figure \ref{fig IncreasinglyManyWellsExample}, where $K_i \rightarrow 0$ as $i \rightarrow \infty$ and which were shown to be annular regions of spheres of radius $\frac{1}{\sqrt{K_i}}$. Outside of these annular regions the metric is rotationally symmetric and the metric becomes essentially Schwarzschild of mass $1/i$. Then by Schoen and Yau \cite{SY2}, and Gromov and Lawson \cite{GL} one can remove arbitrarily small balls, depicted as dotted circles surrounding the gravity wells in figure \ref{fig IncreasinglyManyWellsExample}, such that the metric is unchanged outside these small balls.

If the solution exists beyond the annular regions then by the work of Scheuer \cite{S} the solution should exist for all time and hence we can conclude that $m_H(\Sigma_{T_i}^i) \le \frac{1}{i}$. One issue is that we do not know that $T_i$ will not approach $0$ as $i \rightarrow \infty$, i.e. the smooth solution of IMCF starting at $\Sigma_0^i$ exists for a shorter and shorter time along the sequence (This motivates the need for further exploration of the existence of smooth solutions to IMCF under integral curvature bounds). If $T_i \ge \bar{T} > 0$ then Theorem \ref{SPMTJumpRegion} applies to show SWIF convergence to Euclidean space but one should note that for this very constructive example one can argue the stability directly using an argument similar to Lee and Sormani \cite{LS1} in Example 5.6.

Despite this limitation, this example is particularly instructive for illustrating \eqref{distanceAssumption1} of Theorem \ref{SPMTJumpRegion}. In figure \ref{fig BumpySphere} we see what $\Sigma_0^i$ should look like in this example where we have also drawn a sphere of the same area $S^2_{r_0}$ and a sphere $S^2_{r_i}$ whose distances are smaller than $\Sigma_0^i$, i.e. if $\bar{F}_i= R_i \circ F: \Sigma_0^i \rightarrow S^2_{r_i}$ is a diffeomorphism, constructed from an area preserving diffeomorphism $F: \Sigma_0^i \rightarrow S^2_{r_0}$ and the natural map $R_i: S^2_{r_0} \rightarrow S^2_{r_i}$ which scales the sphere, then we define
\begin{align}
r_i = \sup \{r > 0| d_{\Sigma_0^i}(p,q) \ge d_{S^2_{r_i}}(\bar{F}_i(p),\bar{F}_i(q))\}.
\end{align}
As $i$ increases and the spikes in figure \ref{fig BumpySphere} become thinner the contributed area of the spikes will become smaller and hence $r_i \rightarrow r_0$ as $i \rightarrow \infty$ which implies that condition \eqref{distanceAssumption1} is satisfied.

\begin{figure}[H]\label{fig BumpySphere}
\begin{tikzpicture}
\draw[dotted] (0,0) circle (3cm);
\draw[dotted] (0,0) circle (2cm);
\draw (1.9685,.25) arc (7:82:2cm);
\draw (-.25,1.9685) arc (97:172:2cm);
\draw (-1.9685,-.25) arc (187:262:2cm);
\draw (.25,-1.9685) arc (277:352:2cm);
\draw (-.25,1.9685) .. controls (0,5) .. (.25,1.9685);
\draw (1.9685,-.25) .. controls (5,0) .. (1.9685,.25);
\draw (-.25,-1.9685) .. controls (0,-5) .. (.25,-1.9685);
\draw (-1.9685,-.25) .. controls (-5,0) .. (-1.9685,.25);
\draw (1.6,0) node{$S^2_{r_i}$};
\draw (1,1) node{$\Sigma_0^i$};
\draw (2.8,2.8) node{$S^2(r_0)$};
\end{tikzpicture}
\caption{$\Sigma_0^i$, a sphere of the same area $S^2_{r_0}$, and the largest sphere $S^2_{r_i}$ whose distances are smaller than $\Sigma_0^i$.}
\end{figure}
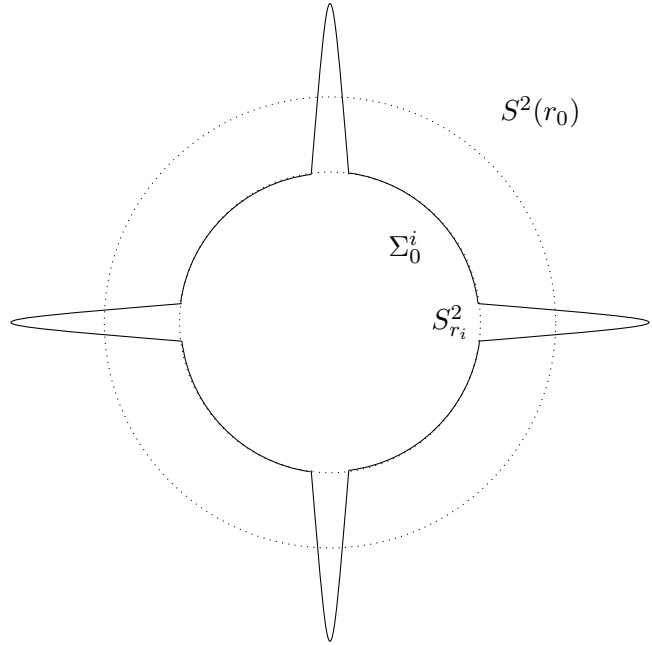
\end{ex}

\begin{ex}\label{Ex WeakConvergence}
In this example we give an illustration of how the main theorems of this paper would combine with results on weak solutions to IMCF. Note that the weak solution of IMCF jumps over gravity wells and so in figure \ref{fig IncreasinglyManyWellsWeakConvergenceExample} we show gravity wells lying on a dotted circle which is separated from the solid region $U^i$ foliated by a weak solution of IMCF. We also depict the region inside $W_i$ and the region outside $V_i$. Note that the weak solution does not foliate the region $M \setminus (V^i\cup U^i \cup W^i)$ and so a separate argument is required to control these jump regions. 

The jump regions become smaller in the parameterizing space as depicted since the weak solution should foliate more of the region as the wells become thinner. The arguments in \cite{LS1} show that the volume of  $M \setminus (V^i\cup U^i \cup W^i)$ will go to zero and hence this example will SWIF converge to Euclidean space on compact subsets but in general the goal would be to control the volume of the jump regions by the mass. If one can show that the volume of the jump region is controlled or going to zero then you can combine with a result like Theorem \ref{SPMTJumpRegion} to show that the entire region is converging to annulus in Euclidean space under SWIF convergence..
\end{ex}

\begin{figure}[H]\label{fig IncreasinglyManyWellsWeakConvergenceExample}
\begin{tikzpicture}[scale=.5]
\draw (-8,0) circle (1.6cm);
\draw (-8,0) node{$W^1$};
\draw (-8,2.7) node{$U^1$};
\draw (-8,0) circle (3.8cm);
\draw (-8,4.6) node{$V^1$};
\draw[dashed] (-8,0) circle (1.8cm);
\draw (-8,0) circle (2cm);
\draw (-8,0) circle (3.4cm);
\draw[dashed] (-8,0) circle (3.6cm);
\draw[fill=black] (-9.8,0) circle (.15cm);
\draw[fill=black] (-6.2,0) circle (.15cm);
\draw[fill=black] (-11.6,0) circle (.15cm);
\draw[fill=black] (-4.4,0) circle (.15cm);
\draw (0,0) circle (1.65cm);
\draw (0,0) node{$W^2$};
\draw (0,2.7) node{$U^2$};
\draw (0,0) circle (3.75cm);
\draw (0,4.6) node{$V^2$};
\draw[dashed] (0,0) circle (1.8cm);
\draw (0,0) circle (1.95cm);
\draw (0,0) circle (3.45cm);
\draw[dashed] (0,0) circle (3.6cm);
\draw[fill=black] (-1.8,0) circle (.1cm);
\draw[fill=black] (1.8,0) circle (.1cm);
\draw[fill=black] (-3.6,0) circle (.1cm);
\draw[fill=black] (3.6,0) circle (.1cm);
\draw[fill=black] (0,-1.8) circle (.1cm);
\draw[fill=black] (0,1.8) circle (.1cm);
\draw[fill=black] (0,-3.6) circle (.1cm);
\draw[fill=black] (0,3.6) circle (.1cm);
\draw (8,0) circle (1.7cm);
\draw (8,0) node{$W^3$};
\draw (8,2.7) node{$U^3$};
\draw (8,0) circle (3.7cm);
\draw (8,4.6) node{$V^3$};
\draw[dashed] (8,0) circle (1.8cm);
\draw (8,0) circle (1.9cm);
\draw (8,0) circle (3.5cm);
\draw[dashed] (8,0) circle (3.6cm);
\draw[fill=black] (8,1.8) circle (.05cm);
\draw[fill=black] (8,3.6) circle (.05cm);
\draw[fill=black] (8,-1.8) circle (.05cm);
\draw[fill=black] (8,-3.6) circle (.05cm);
\draw[fill=black] (6.2,0) circle (.05cm);
\draw[fill=black] (4.4,0) circle (.05cm);
\draw[fill=black] (9.8,0) circle (.05cm);
\draw[fill=black] (11.6,0) circle (.05cm);
\draw[fill=black] (8+1.8*.707,1.8*.707) circle (.05cm);
\draw[fill=black] (8+3.6*.707,3.6*.707) circle (.05cm);
\draw[fill=black] (8-1.8*.707,1.8*.707) circle (.05cm);
\draw[fill=black] (8-3.6*.707,3.6*.707) circle (.05cm);
\draw[fill=black] (8+1.8*.707,-1.8*.707) circle (.05cm);
\draw[fill=black] (8+3.6*.707,-3.6*.707) circle (.05cm);
\draw[fill=black] (8-1.8*.707,-1.8*.707) circle (.05cm);
\draw[fill=black] (8-3.6*.707,-3.6*.707) circle (.05cm);
\draw (13,0) node{$\cdot\cdot\cdot$};
\end{tikzpicture}
\caption{Sequence of manifolds with increasingly thin gravity wells of a fixed depth centered at the points indicated.}
\end{figure}
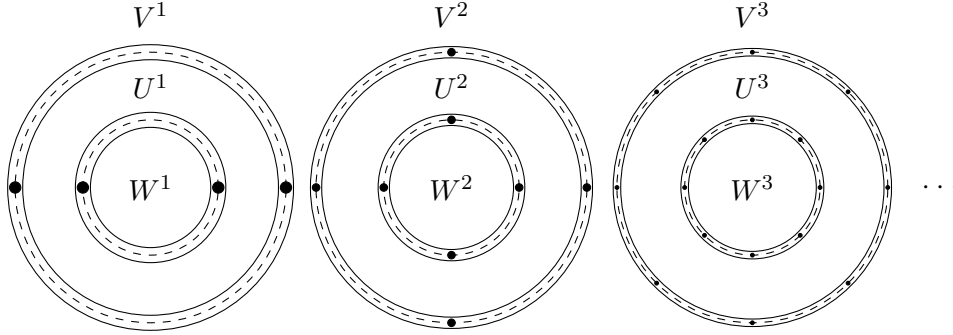

\begin{ex}\label{Ex UsingWarpedProductStability}
In this example we consider four solutions to IMCF which are uniformly controlled along the sequence, i.e. $U_T^i,V_T^i, W_T^i,Z_T^i \subset \mathcal{M}_{r_0,H_0,I_0}^{T,H_1,A_1}$ (See Figure \ref{fig OverlappingRegionsExample}). Here we assume that the Hawking mass of the outer boundary of each region is going to zero and hence by Theorem \ref{SPMTWarped} each region is getting $L^2$ close to a particular warped product. If $U_T^i$ is reaching far enough into the asymptotically flat portion of each $M_i$, where $M_i$ is a uniformly asymptotically flat sequence (See \cite{BA2} for the definition of uniformly asymptotically flat sequence), so that Theorem \ref{SPMTPrevious} applies then we know that $U_T^i$ is converging in $L^2$ to Euclidean space. Now since $U_T^i$ overlaps the other regions $V_T^i, W_T^i, Z_T^i$ we can conclude by uniqueness of limits that the warped product which these regions are converging to is also Euclidean space. 

This suggests that if we can cover a sequence of uniformly asymptotically flat manifolds $M_i$ by a collection of uniformly controlled IMCF coordinate charts then we will be able to show $L^2$ convergence of $M_i$ on compact subsets to Euclidean space. If one assumes further curvature conditions on the coordinate charts then the results of this paper will apply to obtain $C^{0,\alpha}$, GH or SWIF convergence to Euclidean space on compact subsets. This would be analogous to the use of harmonic coordinate charts in the case of smooth Cheeger-Gromov convergence, see \cite{C,G}, under sectional curvature bounds, which was also used by Anderson \cite{A} under Ricci curvature bounds, and has been used by many other authors to develop compactness theorems. 

In order for this to work for IMCF coordinate charts we would need new estimates, under possibly weak integral curvature bounds, which guarantee the existence of smooth, uniformly controlled IMCF charts so that the entire coordinate chart has a uniform lower bound on the existence time or the minimum of the $t$ direction distance. This motivates the pursuit of regularity estimates on smooth and weak IMCF, analogous to the results of Huisken and Ilmanen \cite{HI2}, in order to obtain a better understanding of the existence of IMCF coordinate charts on general manifolds.
\end{ex}

\begin{figure}[H]\label{fig OverlappingRegionsExample}
\begin{tikzpicture}[scale=.5]
\draw[gray] (2.5,2) circle (3.25cm);
\draw[gray] (2.5,2) circle (.25cm);
\draw[dotted] (-2.5,2) circle (3.85cm);
\draw[dotted] (-2.5,2) circle (1cm);
\draw[dashed] (1,-3) circle (4.6cm);
\draw[dashed] (1,-3) circle (1.25cm);
\draw[rotate=45,ultra thin] (0,0) ellipse (3.7cm and 6cm);
\draw[rotate=45,ultra thin] (0,0) ellipse (10cm and 11cm);
\draw (0,8) node{$U_T^i$};
\draw (1,-6) node{$V_T^i$};
\draw (4,3) node{$W_T^i$};
\draw (-1.25,3.75) node{$Z_T^i$};
\end{tikzpicture}
\caption{Collection of uniformly controlled, overlapping IMCF coordinate charts.}
\end{figure}
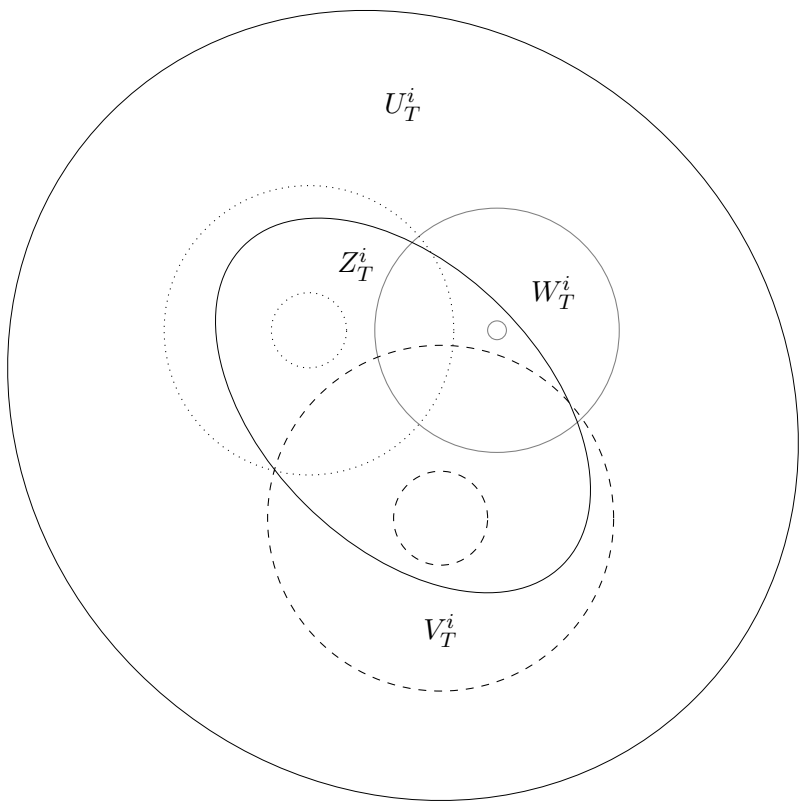

\end{document}